\documentclass[11pt]{amsart}
\usepackage[parfill]{parskip}    
\usepackage{graphicx, txfonts} 
\usepackage{epstopdf}
\usepackage{amsthm, amssymb, amsfonts, amsmath, enumerate}
\usepackage{mathrsfs} 
\usepackage{relsize} 
\usepackage{tikz}

\usepackage{hyperref}

\newcommand{\algt}{\mbox{Alg}_T}

\newcommand{\algtm}{\algt(\M)}

\newcommand{\alg}{\mathsf{Alg}}

\def\Kk{\mathcal{K}}
\def\Ll{\mathcal{L}}

\def\E{\mathscr{E}}

\def\Ee{\mathcal{E}}

\def\Pp{\mathcal{P}}

\def\Set{\mathrm{Set}}

\def\Int{\mathrm{Int}}

\def\Aa{\mathcal{A}}

\def\Tt{\mathcal{T}}
\def\Uu{\mathcal{U}}
\def\Gg{\mathcal{G}}
\def\Ff{\mathcal{F}}

\usepackage{diagrams}
\usepackage[usenames,dvipsnames]{pstricks}
\usepackage{epsfig}
\usepackage{pst-grad} 
\usepackage{pst-plot} 

\usepackage[all,2cell]{xy}
\input xy
\xyoption{all}
\UseTwocells

\usepackage{diagrams}

\usepackage[usenames,dvipsnames]{pstricks}
\usepackage{pst-grad} 
\usepackage{pst-plot} 

\newtheorem{theorem}{Theorem}[section]

\newtheorem{proposition}[theorem]{Proposition}

\newtheorem{lemma}[theorem]{Lemma}

\theoremstyle{definition}
\newtheorem{definition}[theorem]{Definition}
\newtheorem{defn}[theorem]{Definition}
\newtheorem{defin}[theorem]{Definition}

\newtheorem{example}[theorem]{Example}

\newtheorem{pro}[theorem]{Proposition}
\newtheorem{lem}[theorem]{Lemma}

\theoremstyle{remark}
\newtheorem{remark}[theorem]{Remark}
\newtheorem{notation}[theorem]{Notation}

\numberwithin{equation}{section}

\newcommand{\M}{\mathcal{M}}

\newcommand{\calC}{\mathcal{C}}

\newcommand{\Alg}{\mathbb{A}{\rm lg}}
\newcommand{\Cat}{\mathbb{C}{\rm at}}

\newcommand{\fC}{\mathfrak{C}}

\newcommand{\Sigmac}{\Sigma_{\fC}}

\renewcommand{\emptyset}{\varnothing}

\renewcommand{\hat}[1]{\widehat{#1}}

\renewcommand{\bar}[1]{\left. #1 \right|}

\DeclareMathOperator{\Sym}{Sym}

\DeclareMathOperator{\CMon}{CMon}

\newcommand{\algom}{\alg(\cat{O};\M)}

\newcommand{\po}{\ar@{}[dr]|(.7){\Searrow}}
\newcommand{\pb}{\ar@{}[dr]|(.3){\Nwarrow}}

\newcommand{\cat}[1]{\mathcal{#1}}

\newcommand{\boxprod}{\mathbin\square}

\usepackage[margin=1.5in]{geometry}
\usepackage[colorinlistoftodos]{todonotes}

\newcounter{todocounter}

\def\Kk{\mathcal{K}}
\def\Ll{\mathcal{L}}

\def\Ee{\mathcal{E}}

\def\Alg{\mathrm{Alg}}

\def\Pp{\mathcal{P}}

\def\colim{\mathrm{colim}}

\def\Set{\mathrm{Set}}

\def\Cat{\mathrm{Cat}}
\def\CAT{\mathrm{CAT}}
\def\Int{\mathrm{Int}}

\def\Aa{\mathcal{A}}

\def\Uu{\mathcal{U}}
\def\Ff{\mathcal{F}}

\newcommand{\MM}{\mathsf{M}}
\newcommand{\ot}{\leftarrow}
\def\XX{\mathbf{X}}
\def\Int{\mathrm{Int}}

\def\WW{{\mathcal{W}}}
\def\db{{\bar{\Delta}}}

\def\Ii{\mathcal{I}}

\def\Jj{\mathcal{J}}

\newcommand{\h}{{\mathcal P}^{{\mathcal P}_{f,g}}}
\newcommand{\ho}{\mathcal{P}^{\mathcal{P}+2\mathcal{A}}}

\newcommand{\cop}{\mathcal{P}^{\mathcal{P}+\mathcal{A}}}
\newcommand{\ts}{\mbox{$\mathbf{p}$}}

\newcommand{\wa}{\mbox{${\bf w}$}}

\newcommand{\qa}{\mbox{${\bf q}$}}
\newcommand{\la}{\mbox{${\bf l}$}}

\newcommand{\LX}{\mbox{${\bf L}^{}$}}

\newcommand{\HS}{\mbox{${\bf T}^{\scriptstyle \tt S}$}}

\newcommand{\HT}{\mbox{${\bf T}^{{\scriptstyle \tt   T}}$}}

\newcommand{\Sm}{\mbox{${\it S}$}}

{\begin{itemize}\tt}%
{\end{itemize}}

\begin{document}

\title{Quasi-tame substitudes and the Grothendieck construction}


\author{Michael Batanin}
\address{Mathematical Institute of the Academy \\ \v{Z}itn\'a 25, 115~67 Prague 1, Czech Republic}
\email{bataninmichael@gmail.com}

\author{Florian De Leger}
\address{Mathematical Institute of the Academy \\ \v{Z}itn\'a 25, 115~67 Prague 1, Czech Republic}
\email{de-leger@math.cas.cz}

\author{David White}
\address{Department of Mathematics \\ Denison University
\\ Granville, OH 43023}
\email{david.white@denison.edu}
\thanks{The authors were supported by RVO:67985840 and Praemium Academiae of M. Markl.}

\maketitle

\begin{abstract}
This paper continues the study of the homotopy theory of algebras over polynomial monads initiated by the first author and Clemens Berger. We introduce the notion of a quasi-tame polynomial monad (generalizing tame ones) and produce transferred model structures (left proper in many settings) on algebras over such a monad. Our motivating application is to produce model structures on Grothendieck categories, which are used in a companion paper to give a unified approach to the study of operads, their algebras, and their modules. We prove a general result regarding when a Grothendieck construction can be realized as a category of algebras over a polynomial monad, examples illustrating that quasi-tameness is necessary as well as sufficient for admissibility, and an extension of classifier methods to a non-polynomial situation, namely the case of commutative monoids.
\end{abstract}

{\small \tableofcontents}

\section{Introduction}

It is often advantageous to have a model structure on a category of algebras over an operad or higher operad. Such model structures were crucial to the recent proof of a general stabilization theorem that implied a strong form of the Baez-Dolan stabilization hypothesis in a wide range of settings \cite{crm,bous-loc-semi,white-oberwolfach,Reedy-paper}. The existing literature has dozens of papers working out model structures on various flavors of operads and their algebras and modules. In this paper, we unify many seemingly disparate approaches, and study the homotopy theory of (a wide variety of flavors of) operads and their algebras/modules via the Grothendieck construction. That is, we begin with a category $\cat{B}$ (e.g., a category of operads), and a functor $\Phi:\cat{B}^{op}\to \CAT$ to the large category of categories (e.g., $\Phi(O)$ can be the category of $O$-algebras, left $O$-modules, $(O,P)$-bimodules if $\cat{B}$ is the category of pairs of operads, etc.). The idea is to find some conditions under which 
the Grothendieck construction $\int \Phi$ has a transferred model structure.

We assume that $\cat{B}$ is a category of algebras of a polynomial monad. As demonstrated  in \cite{batanin-berger}  this  assumption holds for most flavors of categories of operads.  Our main observation in this paper  is that in  such a case, $\int \Phi$ can also be encoded as a category of algebras over a polynomial monad. Unfortunately, the methods of \cite{batanin-berger} are not enough to produce a transferred model structure on $\int \Phi$ (which we call a global model structure), because this polynomial monad is rarely tame \cite[Definition 6.19]{batanin-berger}. We therefore generalize \cite{batanin-berger}, first to produce filtrations associated to any polynomial monad (rather than only to a tame one), and then we introduce the notion of a quasi-tame polynomial monad $P.$ The notion of quasi-tameness should be considered as a homotopy invariant version of the notion of tameness. We then show that the canonical filtration for a quasi-tame monad decomposes in such a way that the transferred model structure on $P$-algebras exists.   Indeed, we conjecture that a lack of quasi-tameness yields obstructions to the existence of the transferred model structure so that quasi-tameness is really the best possible hypothesis on a polynomial monad, and we provide some evidence in support of this conjecture.

In \cite[Theorem 3.9]{companion} we prove that the existence of the global model structure on $\int \Phi$ implies the existence of so-called horizontal and vertical model structures on $\cat B$ and every $\Phi(O)$. Hence, our main result, the existence of the global model structure in very general settings, allows us to recover almost all known previous results regarding model structures on categories of operads and their categories of algebras and modules, as part of the same general theory. Furthermore, \cite[Theorem 3.18]{companion} proves that, when the global model structure is left/right proper, then so are the horizontal and vertical model structures. For this reason, in Section \ref{sec:quasi-tame}, we prove a theorem demonstrating when a category of algebra over a quasi-tame polynomial monad, such as $\int \Phi$, is left/right proper. In \cite[Section 3.4]{companion}, this properness result is used to deduce rectification results for vertical model structures, again unifying results spread across the existing literature, and also proving new results in contexts where rectification has not been studied.

It is sometimes the case that $\int \Phi$ has only the global \textit{semi}-model structure and not a full model structure. In \cite[Section 3.2]{companion}, we treat this situation carefully. With Proposition \ref{prop:poly-for-Gr} in hand, it is easy to produce global semi-model structures, so in this paper we focus on the more difficult question of the existence of a full model structure. Our approach also works beyond the setting of polynomial monads, as we demonstrate in Section \ref{sec:non-poly}. 

We work in the ambient setting of a cofibrantly generated monoidal model category $\M$, and we transfer the (semi-)model structure on $\int \Phi$ from a suitable category of collections built from $\M$. 
Since the (semi-)model structure on $\int \Phi$ arises via the machinery of transfer, it is automatically cofibrantly generated (and combinatorial if $\M$ is). 

The generality of our approach allows for applications in the theory of operads, algebras, and modules for a wide variety of operads including non-symmetric, symmetric, cyclic, modular, hyperoperads, and twisted modular operads. 
As we previously mentioned, we begin with a polynomial monad that encodes $\int \Phi$, then produce a framework to transfer a model structure to $\int \Phi$, then deduce the existence of model structures on $\cat B$ and $\Phi(O)$'s. Thus, this paper can also be thought of as an enriched extension of \cite{batanin-berger}, since it results in model structures on $\algom$ for operads $O$ valued in $\M$, rather than only operads $O$ arising from Set-valued polynomials.

We prove our main transfer theorem in slightly more generality in Section \ref{sec:quasi-tame}, using the language of substitudes, in order that our work be applicable to the homotopy theory of $n$-operads. Our approach also works beyond the setting of polynomial monads and substitudes, if one is able to produce a transferred model structure on $\int \Phi$ through other means. We illustrate in \cite[Section 3.4]{companion} with an application in the setting of the model category of small categories, where algebras over any finitary monad admit transferred model structures.

We now explain the layout of the paper. We start in Section \ref{sec:preliminaries} with a review of basic definitions and notation that we will use throughout the paper. In Section \ref{sec:Gr(T)}, we prove a general result that allows us to encode the Grothendieck construction $\int \Phi$ as a category of algebras over a polynomial monad. We then provide many examples that we will study throughout the paper. 

In Section \ref{sec:quasi-tame}, we introduce the notion of quasi-tame polynomial monads, and we prove that algebras over quasi-tame monads have transferred model structures (which are, moreover, left proper if the ambient model category $\M$ is well-behaved), and that the monads encoding the Grothendieck construction are often quasi-tame. Following \cite{Reedy-paper} and motivated by applications to $n$-operads, we work in the generality of substitudes, which are equivalent to operads with a category of colors. We recall the details in Section \ref{subsec:substitudes}.

In Section \ref{sec:applications}, we provide several applications of the global model structure, including to operads, opetopic sequences, homotopy algebras, and twisted modular operads. Lastly, in Section \ref{sec:non-poly}, we discuss the case when $\int \Phi$ is encoded by a non-polynomial monad. We show how to produce a model structure in the case of commutative monoids and their modules, and how to prove this model structure is left proper. 

\section{Preliminaries} \label{sec:preliminaries}

We assume the reader is familiar with monoidal model categories \cite{hovey-book} and the basics of colored operads (a.k.a. multicategories) \cite{BM07}. Our main results will have applications in the setting of colored non-symmetric operads, and various flavors of colored symmetric operads. Recall that a symmetric operad $P$ is \textit{constant free} if $P(0)=\emptyset$ is the initial object of $\M$. 
Call $P$ \textit{reduced} (or \textit{0-reduced}) if $P(0)=\ast$ is the terminal object of $\M$. Call $P$ \textit{non-reduced} if it has no restriction on $P(0)$. 
In this section we remind the reader of some less standard notions that we will require.

\subsection{Polynomial monads}

In this section we recall the theory of polynomial monads \cite{batanin-berger}. 
Recall that a natural transformation between two functors is called {\it cartesian} if all naturality squares are pullbacks. A \emph{monad} $T$ on a category with pullbacks is called \emph{cartesian} if $T$ preserves pullbacks and both, the multiplication and the unit of the monad $T$, are cartesian natural transformations.

Polynomial functors have been studied extensively in category theory (e.g. in \cite{kock-poly}). A polynomial monad $T$ on a comma category $\Set/I$ is a monad in the 2-category of overcategories, polynomial functors, and cartesian natural transformations. In particular, this implies that $T$ has a decomposition as $ t_! \circ p_* \circ s^*: \Set/I \to \Set/I$ generated by a polynomial:

\[
\xymatrix{
	I & E \ar[l]_s \ar[r]^p & B \ar[r]^t & I
}
\]

We will often think of $I$ as a set of colors, of $B$ (resp. $E$) as a set of operations (resp. marked operations), of $s$ (resp. $t$) as a source (resp. target) morphism, and of $p$ as a projection. We will always assume the morphism $p$ has finite fibers.

\subsection{Internal algebra classifiers} \label{subsec:classifiers}

In order to transfer a model structure to the category $T$-algebras, we must compute a pushout of the following form:

\begin{align} \label{diagram:pushout}
\xymatrix{
	T(K) \po \ar[r] \ar[d] & T(L) \ar[d] \\
	X \ar[r] & P
}
\end{align}

Our tool for analyzing such pushouts is the method of classifiers in \cite{batanin-berger}, which converts the problem of computing pushouts of $T$-algebras into a problem of internal categories. This method then finds a representing object (a category) for the pushout, and computes the colimit of this category via a discrete final subcategory in the case where $T$ is tame (Section \ref{subsec:tame}). This procedure can be viewed as a categorification of the simplicial bar construction. We now recall the relevant terms.

Any polynomial monad $T$ induces a monad on $\Cat$. Such a monad is called a \emph{$2$-monad}. Strict algebras for this $2$-monad are called {\it categorical $T$-algebras}. As usual, categorical $T$-algebras can either be considered as internal categories in $T$-algebras, or as $T$-algebras in categories. They form a $2$-category with respect to strict categorical $T$-algebra morphisms and $T$-natural transformations.

\begin{defin}
Let $A$ be a categorical $\,T$-algebra. An \textit{internal $T$-algebra} $X$ in $A$ is a lax morphism of categorical $T$-algebras $X:1\to A,$ where $1$ is the terminal categorical $T$-algebra. The internal $T$-algebras in $A$ form a category $\Int_T(A)$ and this correspondence defines a $2$-functor:
$$\Int_T:\Alg_T(\Cat)\to\Cat.$$
\end{defin}

\begin{theorem}[\cite{batanin-eckmann-hilton}]\label{bar1} 
The $2$-functor $\,\Int_T$ is represented by a categorical $T$-algebra $\HT.$ The underlying categorical object of  $\HT$ is the $2$-truncated simplicial object
\begin{equation}\label{HT}
	\xygraph{!{0;(2,0):(0,1)::} {T^31}="p0" [r] {T^21}="p1" [r] {T1}="p2"
		"p2":"p1"|-{T\eta^T_1} "p1":@<1.5ex>"p2"^-{\mu^T_1} "p1":@<-1.5ex>"p2"_-{T(!)} "p0":@<1.5ex>"p1"^-{\mu^T_{T1}} "p0":"p1"|-{T\mu^T_1} "p0":@<-1.5ex>"p1"_-{T^2(!)}}
\end{equation}
of the simplicial bar-resolution $B(T,T,1)_\bullet$ of the terminal categorical $T$-algebra $1$.
\end{theorem}

The categorical $T$-algebra $T^T$ classifies internal $T$-algebras, i.e. every morphism of categorical $T$-algebras $1 \to \calC$ (that picks out internal $T$-algebras in $\calC$) corresponds one-to-one with a morphism $T^T \to \calC$ of categorical $T$-algebras. This categorical $T$-algebra $\HT$ will be called the \emph{internal algebra classifier} of $T$ because of its universal property.

There is an analogous formula in the non absolute case. Namely, for a morphism of polynomial monads $f: S \to T$, which is given on colors by a function $\phi: J \to I$, we have the following commutative square of adjunctions:  
\[
\xymatrix{
	\Alg_S(\Set) \ar@<-.5ex>[rr]_-{f_!} \ar@<-.5ex>[dd]_-{\scriptstyle \mathcal{U}_S} && \Alg_T(\Set) \ar@<-.5ex>[ll]_-{f^*} \ar@<-.5ex>[dd]_-{\scriptstyle \mathcal{U}_T} \\
	\\
	\Set^J \ar@<-.5ex>[rr]_-{\phi_!} \ar@<-.5ex>[uu]_-{\scriptstyle \mathcal{F}_S} && \Set^I \ar@<-.5ex>[ll]_-{\phi^*} \ar@<-.5ex>[uu]_-{\scriptstyle \mathcal{F}_T}
}
\]
Here $\phi^*$ is the  restriction functor induced by $\phi: J \to I$ and $\phi_!: \Set^J\to \Set^I$ is its left adjoint given by coproducts over fibers of $\phi$.

\begin{defin}Let $f: S \to T$ be a morphism of polynomial monads. Let $A$ be a categorical $T$-algebra. An internal $S$-algebra in $A$ is a lax morphism of categorical $S$-algebras $X:1\to f^*(A).$ There is a $2$-functor
	$$\Int_S: \Alg_T(\Cat)\to\Cat$$which associates to $A$ the category of internal $S$-algebras in $A.$ \end{defin}

\begin{theorem}[\cite{batanin-eckmann-hilton}]\label{TST}
The $2$-functor $\,\Int_S$ is represented by a categorical $T$-algebra $\HS.$ The underlying categorical object of $\HS$ is the $2$-truncated simplicial object
	\begin{equation}\label{objects}
	\xygraph{!{0;(2,0):(0,1)::} {\mathcal{F}_T \phi_! S^2(1)}="p0" [r] {\mathcal{F}_T \phi_! S(1)}="p1" [r] {\mathcal{F}_T \phi_!(1)}="p2"
"p2":"p1"|-{} "p1":@<1.5ex>"p2"^-{} "p1":@<-1.5ex>"p2"_-{} "p0":@<1.5ex>"p1"^-{} "p0":"p1"|-{} "p0":@<-1.5ex>"p1"_-{}}
\end{equation}
of the two-sided bar-construction $B(\mathcal{F}_T \cdot \phi_!,S,1)_\bullet$ where $1$ is a terminal categorical $S$-algebra.
\end{theorem}

This categorical $T$-algebra $T^S$ classifies internal $S$-algebras in a categorical $T$-algebra $\calC$, and hence is called the \emph{internal $S$-algebra classifier} of $T$. 

The method of classifiers \cite{batanin-berger} provides filtrations essential for transferring (semi-)model structures to $\algtm$, whenever the monad $T$ is polynomial. We build on this section and generalize those filtrations in Section \ref{sec:quasi-tame} below.

Every polynomial monad $T$ has an associated $\Sigma$-cofibrant set-valued colored operad (with color set $I$) $O_T$ with operations given by elements $b\in B$ and inputs given by $p^{-1}(b)$. While algebras over the monad $T$ are the same as $O_T$-algebras, the language of polynomials makes it easy to prove that the associated operad is $\Sigma$-cofibrant, and gives a very explicit description of the trees used to encode operations. The use of polynomials avoids the need to discuss symmetric group actions; they only arise in the passage to $O_T$. Hence, many constructions (e.g. the plus construction \cite{batanin-kock-joyal}) are simpler with polynomial monads than with colored operads.

\subsection{Substitudes} \label{subsec:substitudes}

In this section, we recall the definition of $\Set$-valued substitudes, which can encode algebraic settings that Set-valued polynomials cannot. For more details, we refer the reader to Section 5 of \cite{Reedy-paper}. Throughout, $\Sigma_n$ denotes the symmetric group on $n$ letters.

\begin{definition} \label{defn:substitude}
A {\em substitude\/}  $(P,A)$ is a small 
category  $A$ together
with a sequence of functors:
$$
P^n: \underbrace{A^{op}\otimes\cdots\otimes A^{op}}_{n-times}\otimes  
A \rightarrow \Set, \  n\ge 0,\ $$ 
equipped with suitable substitution operations, unit morphisms $\eta: A(a_1,a_2)\rightarrow P^1(a_1;a_2);$, and $\Sigma_n$-actions on $P^n(a_1,\ldots,a_n;a)$.

A morphism of substitudes $(P,A)\to (Q,B)$ is a pair
$(f,g)$ where $g:A\to B$ is a functor and $f$ is a sequence of natural transformations
$$f^n:P^n\to g^*(Q^n)$$
and $g^*$ is the restriction functor along $(g^{op})^n\otimes g$ which respects substitution and unit operations. 
\end{definition}

To simplify notation we often omit the superscript in the notation $P^n(a_1,\ldots,a_n; a).$ Notice that $P^1$ is a monad on $A$ in the bicategory of profunctors. The Kleisli category of this monad is called {\it the underlying category of $P.$}

Any colored operad $\E$ in $\Set$  can be naturally considered as a substitude in several different ways, e.g., taking $A$ equal to the category of all unary operations in $\E$ and $P(\E)(a_1,\ldots,a_n,a) = \E(a_1,\ldots,a_n,a).$ The category $A = U(\E)$ is
also called the underlying category of the colored operad $\E.$ In general, a substitude is a colored operad $\E$ together with a small category
$A$ and an identity-on-objects functor $\eta: A \rightarrow U(\E)$. 

Yet another possibility is to consider the full subcategory of substitudes $(P,A)$ for which $A$ is a discrete category. This subcategory  is isomorphic to the category of operads.
This full inclusion functor has a right adjoint that takes a substitude $(P,A)$ to the substitude $(P_0,A_0),$ where $A_0$ is the  maximal discrete subcategory of $A$ and $P_0$ is the restriction of $P$ to $A_0.$ 

Let $\MM A$ be the free strict monoidal category on $A$ and $\Sm A$ be the free strict symmetric monoidal category on $A.$ The universal property of $\MM A$ yields a canonical monoidal functor $\epsilon: \MM A\rightarrow \Sm A.$

\begin{defin} \label{defn:free action} A substitude $({P},A)$ is called $\Sigma$-free if there exists a bimodule 
\begin{equation}\label{E'} {d}(P):(\MM A)^{op} \times A \rightarrow \Set,\end{equation}
such that $P$ is the left Kan extension of $d(P)$ along 
   $$\epsilon\times 1:(\MM A)^{op} \times A \rightarrow (\Sm A)^{op}\times A .$$ 
\end{defin}

\subsection{Transferring model structures} \label{subsec:transfer}

In this section, we recall several relevant model categorical conditions from \cite{batanin-berger} that we shall need, in order to transfer left proper model structures to categories of $T$-algebras. We also generalize several results from \cite{batanin-berger}. Analogously, in Section \ref{sec:quasi-tame}, we generalize the filtration results of \cite{batanin-berger} to hold for all polynomial monads, rather than just tame polynomial monads. The condition below allows us to use the small object argument in categories of algebras.

\begin{defn} \label{defn:compactly-generated}
Let $\cat K$ be a class of morphisms. A cofibrantly generated model category $\M$ is called \textit{$\cat K$-compactly generated}, if all objects are small relative to $\cat K$-cell and if the weak equivalences are closed under filtered colimits along morphisms in $\cat K$. $\M$ is called \textit{compactly generated} if it is $\cat K$-compactly generated for $\cat K$ the monoidal saturation of the generating cofibrations.
\end{defn}

The next condition, due to Schwede and Shipley \cite{SS00}, is required to transfer a model structure to a category of monoids (and, more generally, to algebras over the tame polynomial monads recalled in Section \ref{subsec:tame}).

\begin{defn} \label{defn:monoid-axiom}
Given a class of morphisms $\calC$ in $\M$, let $\calC \otimes \M$ denote the class of morphisms $f \otimes id_X$ where $f\in \calC$ and $X\in \M$. A model category is said to satisfy the \textit{monoid axiom} if every morphism in (Trivial-Cofibrations $\otimes \M)$-cell is a weak equivalence. 
\end{defn}

Recall that a monad $T$ on $\M$ is called \emph{finitary} if $T$ preserves filtered colimits, or equivalently, if the forgetful functor $U_T:\Alg_T\to\M$ preserves filtered colimits. Here, $\Alg_T$ denotes the category of $T$-algebras and $F_T \dashv U_T : \Alg_T \to \M$ the free-forgetful adjunction. Thus $T=U_TF_T$ and $F_T(X)=(TX,\mu_X)$ where $\mu:T^2\to T$ is the multiplication of the monad $T$.

Recall that the theory of left Bousfield localization of model categories requires left properness (though, see \cite{bous-loc-semi} for results in non-left proper settings). In operadic contexts, applications are explored in \cite{batanin-white-eilenberg-moore, white-localization, white-yau, white-yau-co}. In order to prove that transferred model structures are left proper, one must analyze diagrams of the following form, and show that, whenever $u:K\to L$ is a cofibration in $\M$ and $f:A \to B$ is a weak equivalence in $\algtm$, then $A[u,\alpha] \to B[u,f\alpha]$ is a weak equivalence in $\algtm$:
\begin{equation*}\label{diagram:cell-extension-for-left-proper}
\xymatrix{F_T(K) \ar[d]_{F_T(u)} \ar[r]^{F_T(\alpha)} \po & F_T(U(A)) \ar[d] \ar[r] \po & A\ar[d] \ar[r]^f \po & B\ar[d] \\
F_T(L) \ar[r] & F_T(P) \ar[r] & A[u,\alpha] \ar[r] & B[u,f\alpha]}
\end{equation*}
Doing so requires the following definitions. 

\begin{defn} \label{defn:h-cofibration} 
A morphism of free $T$-algebras $F_T(u):F_T(K)\to F_T(L)$ is an \emph{$h$-cofibration} if in any diagram like (\ref{diagram:cell-extension-for-left-proper}), in which $f$ is a weak equivalence, the induced morphism $A[u,\alpha]\to B[u,f\alpha]$ is again a weak equivalence. We shall say that $F_T(u)$ is a \emph{relative $h$-cofibration} if the latter preservation property only holds for those $f$ for which $U_T(A)$ and $U_T(B)$ are cofibrant in $\M$.
\end{defn}

\begin{defn} \label{defn:h-monoidal}
A monoidal model category is \textit{$h$-monoidal} if for each (trivial) cofibration $f:X\to Y$ and each object $Z$, the morphism $f\otimes Z$ is a (trivial) $h$-cofibration, i.e. pushouts along this morphism preserve weak equivalences. Formally, this means that in any diagram as below, in which both squares are pushout squares and $w$ is weak equivalence, then $w'$ is also a weak equivalence:
\[
\xymatrix{X\otimes Z \ar[r] \ar[d]_{f\otimes Z} & A \ar[r]^w \ar[d] & B\ar[d]\\
Y \otimes Z \ar[r] & A' \ar[r]_{w'} & B'}
\]

$\M$ is \textit{strongly $h$-monoidal} if in addition the weak equivalences are closed under tensor product.
\end{defn}

These conditions on $\M$ are verified in \cite{batanin-berger} to hold for effectively all model categories of interest. Additionally, monads for a wide variety of operadic situation are analyzed and proven to be tame. We do not need all the results of \cite{batanin-berger}, but we do wish to note that the monad for monoids is tame, the monad for non-symmetric operads is tame (but we will see that the same is not true for the Grothendieck construction for pairs $(O,A)$ where $O$ is a non-symmetric operad and $A$ is an $O$-algebra), the monad for non-reduced symmetric operads is not tame, and the monad for Com is not even polynomial. 

When the base model category $\M$ satisfies the monoid axiom but is not strongly $h$-monoidal, one can often prove that $\algtm$ is still \textit{relatively left proper}:

\begin{defn} \label{defn:relatively-left-proper}
A model structure on $\algtm$ is \textit{relatively left proper} if, for every weak equivalence $f:R\to S$, and any cofibration $R\to R'$, if $R$ and $S$ forget to cofibrant objects in $\M$ then the pushout $R' \to R' \cup_R S$ is a weak equivalence of $T$-algebras.
\end{defn}

Recall (Proposition 2.12 in \cite{batanin-berger}) that the free $T$-algebra functor takes cofibrations to relative $h$-cofibrations if it takes cofibrations with cofibrant domain to relative $h$-cofibrations.

With these definitions in hand, we are ready to turn to the question of transferred model structures. We include the semi-model category case \cite[Section 2]{Reedy-paper} in the following, because it is needed in \cite{companion}.

\begin{defin} \label{admissible} 
Let $\Ee$ be a model category, $\mathcal{W}$ its class of weak equivalences, and $\Kk,\Ll$ be two saturated class in $\Ee$ such that $\Ll \subset \mathcal{W} \cap \Kk .$  A monad $T$ on $\Ee$ is said to be \emph{$(\Kk,\Ll)$-admissible} if for each cofibration (resp. trivial cofibration) $u:X\to Y$ and each morphism of $T$-algebras $\alpha:F_T(X)\to R$, the pushout in $\Alg_T$
\begin{equation}\label{extension} 
\xygraph{!{0;(2,0):(0,.5)::} 
{F_T(X)}="p0" [r] {R}="p1" [d] {R[u,\alpha]}="p2" [l] {F_T(Y)}="p3" 
"p0":"p1"^-{\alpha}:"p2"^-{u_{\alpha}}:@{<-}"p3"^-{}:@{<-}"p0"^-{F_T(u)} 
"p2" [u(.4)l(.3)] (-[d(.2)],-[r(.15)]) 
}
\end{equation}
yields a $T$-algebra morphism $u_{\alpha}: R\to R[u,\alpha]$ whose underlying morphism $U_T(u_{\alpha})$ belongs to $\Kk$ (resp. to $\Ll$). We say $T$ is \textit{$(\cat K, \Ll)$-semi admissible} (resp. $(\cat K, \Ll)$-semi-admissible over $\M$) if this holds for pushouts into $T$-algebras $R$ that are cofibrant (resp. cofibrant in $\M$). When we say $T$ is \textit{admissible} (resp. semi-admissible) we mean it is $(\cat K, \Ll)$-(semi-)admissible for some $(\cat K, \Ll)$.
\end{defin}

\begin{lem} 
If $\Ee$ is cofibrantly generated then a monad $T$ is $(\Kk,\Ll)$-admissible if and only if the morphism $U_T(u_{\alpha})$ belongs to $\Kk$ (resp. to $\Ll$) for any generating cofibration (resp. generating trivial cofibration) $u.$
\end{lem}

\begin{proof} 
The forward implication is trivial. To prove the inverse implication it is enough to observe that in a cofibrantly generated category any cofibration is a retract of a cellular morphism. Since $F_T$ preserves colimits and $\Kk$ is saturated the result follows. The same argument is valid for trivial cofibrations and class $\Ll.$ 
\end{proof}

The definition above isolates precisely what is needed to transfer a model structure or semi-model structure to the category of $T$-algebras. We recall that a category has a semi-model structure (resp. semi-model structure over $\Ee$) if it satisfies all the model structure axioms, but only admits the (trivial cofibration, fibration) factorization, and the lifting of a trivial cofibration against a fibration, for morphisms with cofibrant domain (resp. domain that becomes cofibrant in $\Ee$) \cite[Definition 2.1]{companion}. We will not need semi-model categories in this paper, but the following result is used in \cite{companion}.

\begin{theorem}\label{thm:admissible-implies-(semi)-models} \label{admissible}
For any finitary $(\Kk,\Ll)$-admissible (resp. semi-admissible, resp. semi-admissible over $\Ee$) monad $T$ on a $\Kk$-compactly generated model category $\Ee$, the category of $T$-algebras admits a transferred cofibrantly generated model structure (resp. semi-model structure, resp. semi-model structure over $\Ee$). This (semi-)model structure is (relatively) left proper if and only if the free $T$-algebra functor takes cofibrations in $\Ee$ to (relative) $h$-cofibrations in $\Alg_T$.
\end{theorem}

\begin{proof}
The proof in case where $T$ is $(\Kk,\Ll)$-admissible follows exactly as in Theorem 2.11 of \cite{batanin-berger} except with $\Ll$ replacing $\mathcal{W} \cap \Kk$. For the cases of semi-admissibility, we verify (ii) of Theorem 2.11 of \cite{batanin-berger} only for morphisms between cofibrant (resp. underlying cofibrant) objects, and this is what the semi-admissibility hypotheses guarantee. For the semi-admissibility cases, left properness is automatic, since it asks that weak equivalences between cofibrant objects be preserved by pushout along cofibrations. In the case of a semi-model structure over $\Ee$, this same proof shows $\Alg_T$ is relatively left proper.
\end{proof}

\begin{defin}
A monad $T$ is \emph{$(\Kk,\Ll)$-adequate} if the underlying morphism of any free $T$-algebra extension $u_\alpha:R\to R[u,\alpha]$ admits a functorial factorization
$$U_T(R)= R[u]^{(0)}\to R[u]^{(1)}\to \ldots \to R[u]^{(n)}\to \ldots \to \colim_n R[u]^{(n)}= U_T(R[u,\alpha]);$$
such that for a cofibration (resp. trivial cofibration) $u$, each morphism of the sequence belongs to $\Kk$ (resp. $\Ll$), and moreover for a weak equivalence $f:R\to S$, the induced morphisms $R[u]^{(n)}\to S[u]^{(n)}$ are weak equivalences for all $n\ge 0.$

The monad $T$ is \emph{relatively $(\Kk,\Ll)$-adequate} if the last property only holds if  $u$ is a cofibration with cofibrant domain and  $f:R\to S$ is a weak equivalence  with cofibrant underlying objects $U_T(R)$ and $U_T(S).$\end{defin}

\begin{theorem}\label{adequate}
Any finitary (relatively) $(\Kk,\Ll)$-adequate monad $T$ on a $\Kk$-compactly generated  model category $\Ee$ is $(\Kk,\Ll)$-admissible, and the associated free $T$-algebra functor takes cofibrations to (relative) $h$-cofibrations. Hence, the category of $T$-algebras has a transferred model structure which is (relatively) left proper.
\end{theorem}

\begin{proof}
The proof proceeds mutatis mutandis from the proof of Theorem 2.23 of \cite{batanin-berger}.
\end{proof}

\begin{remark}
When $T$ is polynomial, it is always semi-admissible, by Theorem 6.3.1 of \cite{white-yau}. It is sometimes semi-admissible over $\Ee$, e.g., when $Alg_T$ is the category of symmetric operads. We will see in Section \ref{sec:quasi-tame} that quasi-tame monads are admissible.
\end{remark}

\subsection{Tame polynomial monads} \label{subsec:tame}

Until the present work, the primary source of admissible monads was tame polynomial monads. We review the relevant terminology, since the first example in the next section is tame.

For a finitary monad $T$ on a cocomplete category $\mathbb{C}$, let $T+1$ be the finitary monad on $\mathbb{C} \times \mathbb{C}$ given by
\begin{equation*}
\begin{gathered}
(T+1)(X,Y) = (T(X),Y) \\
(T+1)(\phi,\psi) = (T(\phi),\psi)
\end{gathered}
\end{equation*}
with evident multiplication and unit. If $T$ is a polynomial monad, so is $T+1$ \cite{batanin-berger}. A polynomial monad $T$ is said to be \emph{tame} if the classifier $T^{T+1}$ is a coproduct of categories with terminal object. In this case, the terminal objects (one in each connected component) form a final discrete subcategory, drastically simplifying colimit computations. Furthermore, when $T$ is tame, the classifier $T^{T_{f,g}}$ described in \cite[Section 7]{batanin-berger} has a final subcategory spanned by a final discrete subcategory of simpler classifiers. This allows for the computation (Theorem 7.11 \cite{batanin-berger}) of (\ref{diagram:pushout}) as $\colim_k P_k$ where each $P_{k-1} \to P_k$ is a pushout in $\M$ of the morphism from the colimit of a punctured cube (whose vertices are words in $K,L,X$, and whose edges convert $K$'s to $L$'s or $X$'s) to the terminal vertex of all $L$'s and $X$'s. This computation is the key input to the following theorem (Theorem 8.1 of \cite{batanin-berger}):

\begin{theorem} \label{thm:tame-admissible}
Suppose $T$ is a tame polynomial monad. Suppose $\M$ is compactly generated and satisfies the monoid axiom. Then $\algtm$ is a relatively left proper model category, which is left proper if $\M$ is strongly $h$-monoidal.
\end{theorem}

In the proof of this theorem, the filtration from \cite{batanin-berger} (generalized in Proposition \ref{filtration} below) is used to filter each morphism of the form $A[u,\alpha] \to B[u,f\alpha]$ into a transfinite sequence of trivial $h$-cofibrations, which will be a weak equivalence as required.

In this paper, we will make heavy use of Proposition \ref{filtration}, and most of the monads we encounter will be polynomial. Numerous examples are given in \cite{batanin-berger}, and more can be found in Sections \ref{sec:quasi-tame} and \ref{sec:non-poly}. The filtration of Proposition \ref{filtration} is fundamental for us because it will enable us to replace free algebra extensions by left Kan extensions which are much easier to analyze. For  tame and quasi-tame polynomial monads, the filtration decomposes into a particularly tractable form, so that $T$-algebras have transferred left proper model structures in many settings of interest \cite{batanin-berger}.

\section{Polynomial monads for Grothendieck constructions} \label{sec:Gr(T)}

In this section, we prove a general result about when $\int \Phi$ is encoded by a polynomial monad, and we give numerous examples of this phenomenon.

\subsection{The Grothendieck construction} \label{subsec:grothendieck-preliminaries}

We record some basic observations about the Grothendieck construction that we will need.

Let $\cat{B}$ be a category and 
$$\Phi:\cat{B}^{op} \to \CAT$$
be a functor.
Let $\phi:O\to O'$ be a morphism in $\cat{B}.$ We will denote  $\phi^* =\Phi(\phi): \Phi(O')\to \Phi(O).$ We can form now the Grothendieck construction $\int \Phi$: the objects are pairs $(O,A)$ where $O\in  \cat{B}$ and $A\in \Phi(O).$ 
A morphism 
$$(O,A)\to (O',A')$$
is a pair $(\phi,f)$ where $\phi:O\to O'$ and $f:A\to \phi^*(A').$ 
We will very often identify a morphism $f$ in $\Phi(O)$ with a morphism $(id,f)$ in $\int \Phi.$ 

As we are interested in model structures, we will assume $\int \Phi$ is cocomplete. This implies that each $\phi^*$ is a right adjoint. The Grothendieck construction always has a projection functor $p:\int \Phi \to \cat{B}$ taking $(O,A) \to O$. If each category $\Phi(O)$ has an initial object $i_O$, then this functor $p$ admits a left adjoint $i:\cat{B} \to \int \Phi$ by $O \mapsto (O,i_O)$. In our settings these will form a Quillen pair, since $p$ preserves (trivial) fibrations by definition of global and horizontal weak equivalences and fibrations. If each category $\Phi(O)$ has a terminal object $t_O$, then this functor $p$ admits a right adjoint $r:\cat{B} \to \int \Phi$ by $O \mapsto (O,t_O)$. In our settings, this functor $r$ is a right Quillen functor, by the definition of global and horizontal structures.

Sometimes, the functor $i$ admits a further left adjoint $E:\int \Phi \to \cat{B}$. This is the case for certain categories of operads and their algebras, and $E(O,A) = O_A$ is the \textit{enveloping operad}, whose algebras are $O$-algebras under $A$ \cite{BM07}. This is also the case for categories of (commutative) monoids and modules, and $E(R,M) = R_M$ is the \textit{enveloping algebra} of the pair $(R,M)$, where $M$ is an $R$-module. Note also that the functor $E$ does not exist for all categories of operads. For example, if we work with reduced operads, then there is no enveloping operad construction, since $O_A$ is not a reduced operad ($O_A(0) \cong A$).

These functors and the adjunctions are summarized in Figure (\ref{figure:adjoints}):

\begin{align} \label{figure:adjoints}
\xymatrix{
	\int \Phi
	\ar@<0.75cm>@{}[d]|{\dashv} 
	\ar@/^0.5cm/[d]_{p} 
	\ar@{}[d]|{\dashv} 
	\ar@<-0.75cm>@{}[d]|{\dashv} 
	\ar@/_1cm/[d]_{E} 
	\\
	\cat{B}
	\ar@/_1cm/[u]_{r} 
	\ar@/^0.5cm/[u]_{i} 
}
\end{align}

We finish this section by proving that, if $\cat{B}$ is the category of algebras over a polynomial monad $T$, and if we have a way to speak of algebras over every $O\in \cat B$, via a morphism of polynomials, then $\int \Phi$ is encoded as algebras over a polynomial monad that we can explicitly construct. The special case where $\cat{B}$ is the category of one-colored symmetric operads, and $\Phi(O) = \Alg_O$, \cite{gutierrez-rondigs-spitzweck-ostvaer-colocalizations} (Lemma 4.7) proves that there is a $\Sigma$-cofibrant colored operad for $\int \Phi$, making use of results in \cite{BM07}. 

Let $T$ be an $I$-colored symmetric operad in a closed symmetric monoidal category $\M$, equipped with a morphism of operads $\phi: T \to SOp(J)$, where $SOp(J)$ is the $\Sigma$-free symmetric operad for $J$-colored symmetric operads. This morphism $\phi$ induces a restriction functor
$$\phi^*: \Alg_{SOp(J)}(\M) \to \Alg_T(\M).$$
This restriction functor allows us to talk about algebras of $O$, when $O$ is itself an algebra of $T$:

\begin{definition}
Let $O$ be an algebra of $T$. An algebra of $O$ in $\M$ is a $J$-collection $C=\{C_j \;|\; j\in J\}$ of objects of $\M$ equipped with a morphism of $T$-algebras
$$O\to \phi^*(End(C))$$
where $End(C)$ is the endomorphism operad of $C.$
\end{definition}

One can define a morphism of $O$-algebras in the usual way. If $\M$ is cocomplete then $\phi^*$ admits a left adjoint $\phi_!$ and the category of $O$-algebras is isomorphic to the category  $\Alg_{\phi_!(O)}(\M).$  

We now have a functor 
$$\Phi:\Alg_{T}(\M)^{op}\to \Cat$$
which assigns to a $T$-algebra $O$ the category of $O$-algebras and to a morphism of $T$-algebras $f:O\to O'$ assigns the restriction
functor $f^*: \Alg_{O'}(\M)\to \Alg_O(\M).$ We can form now the Grothendieck construction $\int \Phi.$ So, the objects of $\int \Phi$ are pairs $(O,A)$ where $O\in \Alg_T(\M)$ and $A\in \Alg_O(\M).$ 
A morphism $$(O,A)\to (O',A')$$ is a pair $(f,\alpha)$ where $f:O\to O'$ and $\alpha:A\to f^*(A').$ 

Let now $\M$ be a symmetric monoidal model category. We will call a morphism $(f,\phi)$ in $\int \Phi $ a weak equivalence (fibration) if the underlying morphisms of $I$-collections
$U(f)$ and $J$-collections $U(\phi)$ are pointwise weak equivalences (fibrations).
We will say that $\int \Phi$ admits the {\em global model structure} if there is a model structure on $\int \Phi$ with weak equivalences and fibrations defined as above \cite{harpaz-prasma-integrated}.

It is in general a difficult task to know if such a model structure exists. In this paper we concentrate on the case when $T$ comes from a polynomial monad in $\Set$, and, moreover, $\phi$ is a morphism of polynomial monads (using that $SOp(J)$ is a polynomial monad). 

\begin{proposition} \label{prop:poly-for-Gr}
If $\phi$ is a morphism of polynomial monads in $\Set$ then there exists a polynomial monad $Gr(T)$ such that the category $\int \Phi $ is isomorphic to the category $\Alg_{Gr(T)}(\M).$ 
\end{proposition} 

\begin{proof}
Recall from \cite[Section 9.4]{batanin-berger} that $SOp(J)$ is represented by the polynomial 
$$Bq(J) \leftarrow ORTr^*(J) \to ORTr(J) \to Bq(J)$$
where $ORTr(J)$ is the set of ordered rooted trees whose edges are decorated by $J$ and $Bq(J)$ is the set of decorated bouquets. $ORTr^*(J)$ is the set of decorated ordered rooted trees with a marked vertex. Let $T$ be represented by a polynomial $I \leftarrow E \to B \to I$ and $\phi$ is given by the morphism of polynomials

\begin{diagram}
I & \lTo^{s} & E\SEpbk &\rTo^{p} & B &\rTo^{t}&I\\
\dTo^{c}&&\dTo&&\dTo_{\psi}&&\dTo_{c}\\
Bq(J)&\lTo& ORTr^*(J) &\rTo & ORTr(J)&\rTo & Bq(J)\end{diagram}

We define a polynomial monad 
\begin{diagram}I\sqcup J &\lTo^{\xi}&D^*&\rTo^{\pi}&D&\rTo^{\tau}&I\sqcup J\end{diagram}
in the following way. The set $D$ is the coproduct $$B \sqcup B = B \sqcup \{(b,\sigma)\ | \ \sigma \in ORTr(J), b\in \psi^{-1}(\sigma)\}.  $$ 
The target of an element $b\in B$ is $t(b)$ and the target of $(b,\sigma)$ is the target of $\sigma.$ 
	
The set $D^*$ is the set of elements of $D$ with a marked source in the sense that given $b\in B$ we mark one of the element of $p^{-1}(b).$ For a pair $(b,\sigma)$ we either mark an element of $p^{-1}(b)$ or we mark one of the leaves of $\sigma.$ By forgetting the marked elements we obtain a morphism $D^*\to D.$ 
	
	Finally, the source of the element from $D^*$ coincides with the source of the marked element of $p^{-1}(b)$ or the color of the marked leaf of $\sigma.$ 
	
	The substitution operation is given by substitution in $T$ if the source was from $I$, or grafting of trees if the source was from $J.$ 
\end{proof}

In Section \ref{sec:Gr(T)} we will see that, if $T$ is the monad for monoids, then the polynomial monad for $Gr(T)$ is tame. In Section \ref{sec:quasi-tame}, we will see that, if $T$ is the monad for non-symmetric operads then the monad $Gr(T)$, whose algebras are pairs $(O,A)$ where $A$ is an $O$-algebra, is quasi-tame but is not tame. The same is true for the monad for pairs $(O,M)$ where $O$ is a non-symmetric operad and $M$ is an $O$-module. Furthermore, if $T$ is the monad for reduced symmetric operads, then $Gr(T)$ is not even quasi-tame, and hence we are forced into the world of semi-model categories, treated in \cite[Section 3.2]{companion}. In Section \ref{sec:non-poly}, we will see that, if $T$ is the monad for commutative monoids, then $T$ and $Gr(T)$ are not even polynomial, but we still have a technique to endow their categories of algebras with transferred model structures.

\subsection{Example: monoids and modules} \label{subsec:monoids-modules}

In this section we introduce a polynomial monad $Gr(Mon)$ whose algebras are pairs $(R,M)$ where $R$ is a monoid in $\M$ and $M$ is an $R$-module. 
As a symmetric operad it was described in \cite{BM07}[Example 1.5.1]. We give polynomial description of it and show that this monad is tame in the sense of Batanin-Berger. It follows immediately from \cite{batanin-berger} that under good conditions the category of $Gr(Mon)$-algebras has a model structure, which is (relative) left proper if $\M$ is $h$-monoidal. 
The conditions we obtained are weaker then the conditions from \cite{harpaz-prasma-integrated} for the existence of this model structure.

The polynomial which represents this monad is given by  
\begin{diagram}\{r,m\} &\lTo^{s}&D^*&\rTo^{p}&D&\rTo^{t}&\{r,m\}\end{diagram}
Here $\{r,m\}$ is two element set. The set $D$ consists of linear graphs of two types: 

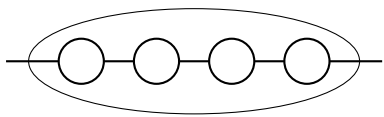
\[
\begin{tikzpicture}
\draw[thick] (-1,0) -- (4,0);
\draw (1.5,0) ellipse (2.2 and .7);
\draw[thick,fill=white] (0,0) circle (.3);
\draw[thick,fill=white] (1,0) circle (.3);
\draw[thick,fill=white] (2,0) circle (.3);
\draw[thick,fill=white] (3,0) circle (.3);
\end{tikzpicture}
\]

and 

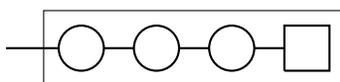
\[
\begin{tikzpicture}
\draw[thick] (-1,0) -- (3,0);
\draw (4,0);
\draw (-.5,-.5) -- (3.5,-.5) -- (3.5,.5) -- (-.5,.5) -- (-.5,-.5);
\draw[thick,fill=white] (0,0) circle (.3);
\draw[thick,fill=white] (1,0) circle (.3);
\draw[thick,fill=white] (2,0) circle (.3);
\draw[thick,fill=white] (2.7,-.3) -- (3.3,-.3) -- (3.3,.3) -- (2.7,.3) -- (2.7,-.3);
\end{tikzpicture}
\]

The target of a graph of the first type is $r$ and the target of the graph of the second type is $m.$ The set $D^*$ consists of the elements of $D$ with one vertex marked. The source of such an element is $r$ if the marked vertex is a circle and it is $m$ if the marked vertex is a box.

As usual, the morphism $p$ forgets the marking and the substitution operation is defined by insertion at the marked vertex. Notice that a circled tree can be inserted only in a circled vertex and a boxed tree can be inserted only in a boxed vertex. Pictorially, this means the insertion operation is the gluing of the outside boundary of a tree to the boundary of the vertex. 

It is not hard to check that the polynomial above defines a polynomial monad $Gr(Mon)$ whose algebras consists of pairs $(R,M)$ consisting of a monoid $R$ and a left $R$-module $M$. We now prove that this monad is tame.

\begin{proposition} \label{prop:Gr(Mon)-tame}
The polynomial monad $Gr(Mon)$ is tame.
\end{proposition}

\begin{proof}
For this we need to describe the classifier $Gr(Mon)^{Gr(Mon)+1}$ (see \cite{batanin-berger}.) According to the general procedure given in \cite{batanin-berger}[6.19] the objects of this classifier are graphs like above with an additional decoration of each vertex by one of two letters $X$ or $K.$ The morphisms are given by contraction of edges which connect two vertices decorated by $X$ with the agreement that a contraction of an edge whose right vertex is a box results to a boxed $X$-vertex. We also can  introduce a new circled $X$-vertex on any edge except for the most right edge of a boxed vertex.

It is now trivial to see that such a category has a terminal object in each connected component. They are precisely linear graphs of the following types:  

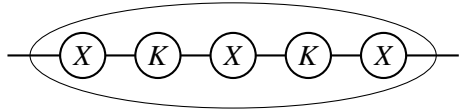
\[
\begin{tikzpicture}
\draw[thick] (-1,0) -- (5,0);
\draw (2,0) ellipse (2.7 and .7);
\draw[thick,fill=white] (0,0) circle (.3) node{$X$};
\draw[thick,fill=white] (1,0) circle (.3) node{$K$};
\draw[thick,fill=white] (2,0) circle (.3) node{$X$};
\draw[thick,fill=white] (3,0) circle (.3) node{$K$};
\draw[thick,fill=white] (4,0) circle (.3) node{$X$};
\end{tikzpicture}
\]

\noindent (see the description of the classifier $T^{T+1}$ from \cite{batanin-berger}) or 

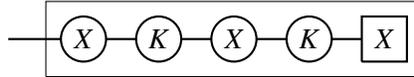
\[
\begin{tikzpicture}
\draw[thick] (-1,0) -- (4,0);
\draw (-.5,-.5) -- (4.5,-.5) -- (4.5,.5) -- (-.5,.5) -- (-.5,-.5);
\draw[thick,fill=white] (0,0) circle (.3) node{$X$};
\draw[thick,fill=white] (1,0) circle (.3) node{$K$};
\draw[thick,fill=white] (2,0) circle (.3) node{$X$};
\draw[thick,fill=white] (3,0) circle (.3) node{$K$};
\draw[thick,fill=white] (3.7,-.3) -- (4.3,-.3) -- (4.3,.3) -- (3.7,.3) -- (3.7,-.3);
\draw (4,0) node {$X$};
\end{tikzpicture}
\]

or 

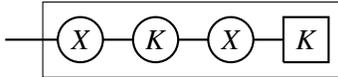
\[
\begin{tikzpicture}
\draw[thick] (-1,0) -- (3,0);
\draw (-.5,-.5) -- (3.5,-.5) -- (3.5,.5) -- (-.5,.5) -- (-.5,-.5);
\draw[thick,fill=white] (0,0) circle (.3) node{$X$};
\draw[thick,fill=white] (1,0) circle (.3) node{$K$};
\draw[thick,fill=white] (2,0) circle (.3) node{$X$};
\draw[thick,fill=white] (2.7,-.3) -- (3.3,-.3) -- (3.3,.3) -- (2.7,.3) -- (2.7,-.3);
\draw (3,0) node {$K$};
\end{tikzpicture}
\]

This shows that the polynomial monad $Gr(Mon)$ is tame.
\end{proof}

\begin{remark} 
Similar calculations show that there exists a tame polynomial monad for monoids and right modules as well as for monoids and right-left modules and for the category of bimodules. As a consequence, these categories inherit transferred model structures.  
\end{remark} 

More generally, one can construct a polynomial monad $Gr(Mon(I))$ whose category of algebras is the Grothendieck construction for the functor
$$Cat(I)^{op}\to \CAT$$
which associates to each small $\M$-category $C$ with set of objects $I$, the category of $\M$-valued presheaves on $C.$ For the construction of this monad it is enough to use linear graphs whose edges are colored by elements of $I$ with the following restriction:     the colors of the outgoing and incoming edges of a boxed vertex must coincide. The set of colors of this monad is $I\times I \sqcup I.$ The target of a linear graph is the pair of colors of roots and the color of the leave for nonboxed graphs and is the color of root for boxed graphs.  The morphisms in the classifier $Gr(Mon(I))^{Gr(Mon(I))+1}$ 
are very much similar to the one-colored case with one extra rule that the result of contraction of an internal edge between a circled vertex and a boxed vertex is a boxed vertex whose edges have colors of the incoming circled vertex. For example, the contraction of the internal edge of the graph 

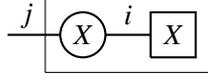
\[
\begin{tikzpicture}
\draw[thick] (-1,0) -- (1.2,0);
\draw (-.75,-.05) node[above]{$j$};
\draw (-.5,-.5) -- (1.7,-.5) -- (1.7,.5) -- (-.5,.5) -- (-.5,-.5);
\draw[thick,fill=white] (0,0) circle (.3) node{$X$};
\draw (.6,0) node[above]{$i$};
\draw[thick,fill=white] (.9,-.3) -- (1.5,-.3) -- (1.5,.3) -- (.9,.3) -- (.9,-.3);
\draw (1.2,0) node {$X$};
\end{tikzpicture}
\]

results in the graph

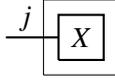
\[
\begin{tikzpicture}
\draw[thick] (-1,0) -- (0,0);
\draw (-.75,-.05) node[above]{$j$};
\draw (-.5,-.5) -- (.5,-.5) -- (.5,.5) -- (-.5,.5) -- (-.5,-.5);
\draw[thick,fill=white] (-.3,-.3) -- (.3,-.3) -- (.3,.3) -- (-.3,.3) -- (-.3,-.3);
\draw (0,0) node {$X$};
\end{tikzpicture}
\]

Obviously this monad is tame for the same reason as in the one-colored case. Hence, by \cite[Theorem 8.1]{batanin-berger}, if $\M$ is a compactly-generated monoidal model category satisfying the monoid axiom, then the Grothendieck construction $\int \Phi$, whose objects are pairs $(R,M)$ where $R$ is a monoid in $\M$ and $M$ is a left $R$-module, possesses a global model structure. If in addition $\M$ is strongly $h$-monoidal, then this model structure is left proper \cite[Theorem 2.11]{batanin-berger}. If $\M$ is only cofibrantly generated, then we still have the global semi-model structure on $\int \Phi$ \cite[Theorem 6.3.1]{white-yau}.

\begin{remark}
The work in this section also proves that, if $A$ is a commutative monoid, then the category of pairs $(R,M)$ where $R$ is an $A$-algebra and $M$ is an $R$-module, possesses the global model structure. For this setting, we work with $A$-modules as our base model category and note that, by \cite[Theorem 4.1]{SS00}, this category of $A$-modules satisfies the pushout product axiom and monoid axiom.
\end{remark}

\begin{remark}
Generalizing the example in this section, one can ask about the category of pairs $(R,M)$ where $R$ is a $P$-algebra and $M$ is an $R$-module, where $P$ is some operad. If $P$-algebras are encoded by a polynomial monad, then Proposition \ref{prop:poly-for-Gr} tells us that this category of pairs is also given by a polynomial monad. However, this polynomial monad will in general not be tame, for $P \neq Ass$. Nevertheless, we still have the global semi-model structure on pairs. 
\end{remark}

\subsection{Example: non-symmetric operads and their algebras}

Our next example is the monad for the Grothendieck construction for the functor $$NO(I)^{op}\to \CAT$$ which associates to an $I$-colored non-symmetric operad its category of algebras. We denote $NOp(I)$ the polynomial monad for $I$-colored non-symmetric operads and by $Gr(NOp(I))$ the corresponding monad for its Grothendieck construction.

The polynomial for $Gr(NOp(I))$ looks like   
\begin{diagram}PBq(I)\sqcup I &\lTo^{s}&PTr^*(I)&\rTo^{p}&PTr(I)&\rTo^{t}&PBq(I)\sqcup I\end{diagram}
where $PBq(I)$ are are planar $I$-bouquets that is $I$-colored planar corollas. A typical picture of the first three corollas is like this:

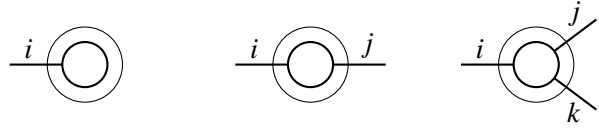
\[
\begin{tikzpicture}
\draw[thick] (-1,0) -- (0,0);
\draw (-.75,-.05) node[above]{$i$};
\draw (0,0) circle (.5);
\draw[thick,fill=white] (0,0) circle (.3);

\begin{scope}[shift={(3,0)}]
\draw[thick] (-1,0) -- (1,0);
\draw (-.75,-.05) node[above]{$i$};
\draw (.75,-.05) node[above]{$j$};
\draw (0,0) circle (.5);
\draw[thick,fill=white] (0,0) circle (.3);
\end{scope}

\begin{scope}[shift={(6,0)}]
\draw[thick] (-1,0) -- (0,0);
\draw[thick] (0,0) -- (.8,-.6);
\draw[thick] (0,0) -- (.8,.6);
\draw (-.75,-.05) node[above]{$i$};
\draw (.5,.375) node[above]{$j$};
\draw (.5,-.375) node[below]{$k$};
\draw (0,0) circle (.5);
\draw[thick,fill=white] (0,0) circle (.3);
\end{scope}
\end{tikzpicture}
\]

The set $PTr(I)$ is the set of isomorphism classes of planar trees whose edges are colored by elements of $I$ and whose vertices can be circled or boxed. Moreover, the entire tree itself is either boxed or circled. The rules are the following. If a tree is circled then all the vertices of this tree are circled. So the boxed vertices only appear on a boxed tree. The second rule is that such boxed vertices have valency one and are always on the top of the tree. Here is an example of a boxed tree.

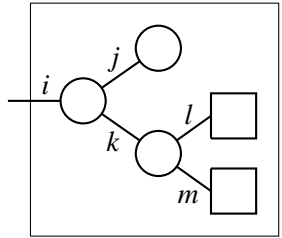
\[
\begin{tikzpicture}
\draw (-.7,-1.8) -- (2.6,-1.8) -- (2.6,1.3) -- (-.7,1.3) -- (-.7,-1.8);
\draw[thick] (0,0) -- (-1,0);
\draw[thick] (0,0) -- (1,.7);
\draw[thick] (0,0) -- (1,-.7);
\draw[thick] (1,-.7) -- (1.7,-.2);
\draw[thick] (1,-.7) -- (1.7,-1.2);
\draw[thick,fill=white] (0,0) circle (.3);
\draw[thick,fill=white] (1,.7) circle (.3);
\draw[thick,fill=white] (1,-.7) circle (.3);
\draw[thick,fill=white] (1.7,-.5) -- (2.3,-.5) -- (2.3,.1) -- (1.7,.1) -- (1.7,-.5);
\draw[thick,fill=white] (1.7,-.9) -- (2.3,-.9) -- (2.3,-1.5) -- (1.7,-1.5) -- (1.7,-.9);
\draw (-.5,-.05) node[above]{$i$};
\draw (.4,.25) node[above]{$j$};
\draw (.4,-.275) node[below]{$k$};
\draw (1.4,-.45) node[above]{$l$};
\draw (1.4,-1.05) node[below]{$m$};
\end{tikzpicture}
\]

As usual, the set $PTr^*(I)$ is the set of planar trees as above with one vertex marked. The morphism $p$ forgets the marking. The morphism $t$ on a non-boxed tree contracts the tree to the colored corolla. The target of a boxed tree is the color of the root of the tree. The source of the marked tree is the colored corolla that corresponds to the vertex if the vertex is circled and it is the color of corresponding leaf if it is boxed. The substitution operation is tree insertion to a circled vertex or a grafting of a tree to the leaf of a marked boxed vertex and removing this boxed vertex (this can be considered as another insertion operation to the marked vertex). In particular, a substitution of a boxed tree with stumps results in the formation of virtual boxed vertices. For example, the result of a substitution 

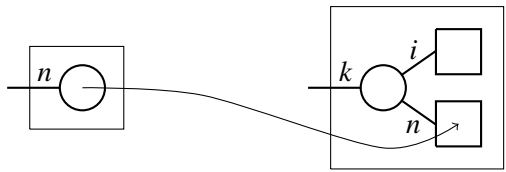
\[
\begin{tikzpicture}
\draw[thick] (0,0) -- (-1,0);
\draw[thick,fill=white] (0,0) circle (.3);
\draw (-.5,-.05) node[above]{$n$};
\draw (-.7,-.55) -- (.55,-.55) -- (.55,.55) -- (-.7,.55) -- (-.7,-.55);
\begin{scope}[shift={(4,0)}]
\draw (-.7,-1.08) -- (1.6,-1.08) -- (1.6,1.08) -- (-.7,1.08) -- (-.7,-1.08);
\draw[thick] (0,0) -- (-1,0);
\draw[thick] (0,0) -- (1,.7);
\draw[thick] (0,0) -- (1,-.7);
\draw[thick,fill=white] (0,0) circle (.3);
\draw[thick,fill=white] (.7,-.78) -- (1.3,-.78) -- (1.3,-.18) -- (.7,-.18) -- (.7,-.78);
\draw[thick,fill=white] (.7,.78) -- (1.3,.78) -- (1.3,.18) -- (.7,.18) -- (.7,.78);
\draw (-.5,-.05) node[above]{$k$};
\draw (.4,.25) node[above]{$i$};
\draw (.4,-.3) node[below]{$n$};
\end{scope}
\draw[->,thin] plot [smooth] coordinates {(0,0) (1.6,-.1) (4,-.8) (5,-.48)};
\end{tikzpicture}
\]

is

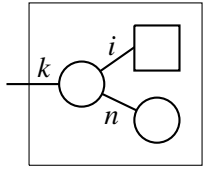
\[
\begin{tikzpicture}
\draw (-.7,-1.08) -- (1.6,-1.08) -- (1.6,1.08) -- (-.7,1.08) -- (-.7,-1.08);
\draw[thick] (0,0) -- (-1,0);
\draw[thick] (0,0) -- (1,.7);
\draw[thick] (0,0) -- (1,-.48);
\draw[thick,fill=white] (0,0) circle (.3);
\draw[thick,fill=white] (.7,.78) -- (1.3,.78) -- (1.3,.18) -- (.7,.18) -- (.7,.78);
\draw[thick,fill=white] (1,-.48) circle (.3);
\draw (-.5,-.05) node[above]{$k$};
\draw (.4,.25) node[above]{$i$};
\draw (.4,-.225) node[below]{$n$};
\end{tikzpicture}
\]

Accordingly, the classifiers $Gr(NOp(I))^{Gr(Nop(I))+1}$ will have objects colored planar trees as above with additional coloring of vertices by letters $X$ or $K$. The morphisms are contractions of  internal edges which connect circled $X$-vertices and introducing a new circled $X$-vertex of valency one. Another type of generating morphism is a simultaneous contraction of a group of incoming edges which connect a circled $X$-vertex to  boxed $X$-vertices. This is possible only if all incoming edges of the circled $X$-vertex are connected to boxed $X$-vertices. In this case we obtain a single boxed $X$-vertex after contraction.  For example,

\begin{equation}\label{morphismalgebra}
\begin{tikzpicture}
\draw (-.7,-1.08) -- (1.6,-1.08) -- (1.6,1.08) -- (-.7,1.08) -- (-.7,-1.08);
\draw[thick] (0,0) -- (-1,0);
\draw[thick] (0,0) -- (1,.7);
\draw[thick] (0,0) -- (1,-.7);
\draw[thick,fill=white] (0,0) circle (.3) node{$X$};
\draw[thick,fill=white] (.7,-.78) -- (1.3,-.78) -- (1.3,-.18) -- (.7,-.18) -- (.7,-.78);
\draw[thick,fill=white] (.7,.78) -- (1.3,.78) -- (1.3,.18) -- (.7,.18) -- (.7,.78);
\draw (1,-.48) node{$X$};
\draw (1,.48) node{$X$};
\draw (-.5,-.05) node[above]{$k$};
\draw (.4,.25) node[above]{$m$};
\draw (.4,-.3) node[below]{$n$};
\draw (2.9,0) node[]{$\longrightarrow$};
\begin{scope}[shift={(5,0)}]
\draw[thick] (0,0) -- (-1,0);
\draw[thick,fill=white] (-.3,-.3) -- (.3,-.3) -- (.3,.3) -- (-.3,.3) -- (-.3,-.3);
\draw (0,0) node{$X$};
\draw (-.5,-.05) node[above]{$k$};
\draw (-.7,-.55) -- (.55,-.55) -- (.55,.55) -- (-.7,.55) -- (-.7,-.55);
\end{scope}
\end{tikzpicture}
\end{equation}

We also have a morphism 

\begin{equation}\label{morphismzerocircletozerobox}
\begin{tikzpicture}
\draw[thick] (0,0) -- (-1,0);
\draw[thick,fill=white] (0,0) circle (.3) node{$X$};
\draw (-.45,-.05) node[above]{$i$};
\draw (-.7,-.55) -- (.55,-.55) -- (.55,.55) -- (-.7,.55) -- (-.7,-.55);
\draw (1.7,0) node[]{$\longrightarrow$};
\begin{scope}[shift={(4,0)}]
\draw[thick] (0,0) -- (-1,0);
\draw[thick,fill=white] (-.3,-.3) -- (.3,-.3) -- (.3,.3) -- (-.3,.3) -- (-.3,-.3);
\draw (0,0) node{$X$};
\draw (-.5,-.05) node[above]{$i$};
\draw (-.7,-.55) -- (.55,-.55) -- (.55,.55) -- (-.7,.55) -- (-.7,-.55);
\end{scope}
\end{tikzpicture}
\end{equation}

This monad is not tame, as we now show.

\begin{example}
Let $I$ be a non-empty set of colors. The monad $T=Gr(NOp(I))$, whose algebras are pairs $(O,A)$ where $O$ is a non-symmetric $I$-colored operad and $A$ is an $O$-algebra, is not tame. The final subcategory of $T^{T_{f,g}}$ cannot be discrete, because it contains zigzags of the form

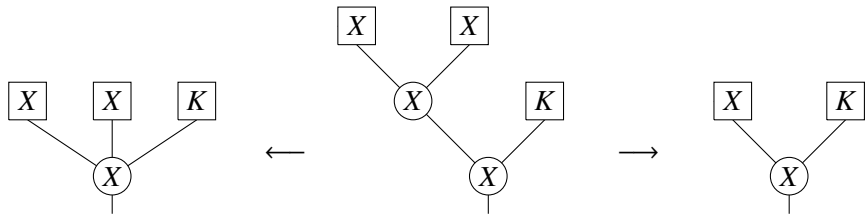
\def\squarex#1#2{
	\draw[fill=white] (#1-.25,#2-.25) -- (#1+.25,#2-.25) -- (#1+.25,#2+.25) -- (#1-.25,#2+.25) -- (#1-.25,#2-.25);
	\draw (#1,#2) node{$X$};
}

\def\squarek#1#2{
	\draw[fill=white] (#1-.25,#2-.25) -- (#1+.25,#2-.25) -- (#1+.25,#2+.25) -- (#1-.25,#2+.25) -- (#1-.25,#2-.25);
	\draw (#1,#2) node{$K$};
}

\[
\begin{tikzpicture}
\draw (0,-.5) -- (0,0);
\draw (0,0) -- (-1.125,.75);
\draw (0,0) -- (0,1);
\draw (0,0) -- (1.125,.75);

\draw[fill=white] (0,0) circle (.25) node{$X$};
\squarex{-1.125}{1};
\squarex{0}{1};
\squarek{1.125}{1};

\draw (2.3,.3) node[]{$\longleftarrow$};

\begin{scope}[shift={(5,0)}]
\draw (0,-.5) -- (0,0);
\draw (0,0) -- (-1,1);
\draw (0,0) -- (1,1);
\draw (-1,1) -- (0,2);
\draw (-1,1) -- (-2,2);

\draw[fill=white] (0,0) circle (.25) node{$X$};
\draw[fill=white] (-1,1) circle (.25) node{$X$};
\squarex{-1.75}{2};
\squarex{-.25}{2};
\squarek{.75}{1};
\end{scope}

\draw (7,.3) node[]{$\longrightarrow$};

\begin{scope}[shift={(9,0)}]
\draw (0,-.5) -- (0,0);
\draw (0,0) -- (-1,1);
\draw (0,0) -- (1,1);

\draw[fill=white] (0,0) circle (.25) node{$X$};
\squarex{-.75}{1};
\squarek{.75}{1};
\end{scope}
\end{tikzpicture}
\]
\end{example}

Fortunately, this $T$ is quasi-tame, by Proposition \ref{propositiongrnopquasitame}, so we will still obtain a model structure on $\int \Phi$. This example was the motivation for our work on quasi-tameness in Section \ref{sec:quasi-tame}.

\subsection{Example: non-symmetric operads and their left modules} \label{subsec:Gr(Op,Mod)}

In order to study the homotopy theory of operadic modules and bimodules, it is important to study the category of pairs $(O,M)$, where $O$ is a non-symmetric operad and $M$ is a left $O$-module. We cannot apply Proposition \ref{prop:poly-for-Gr} to this situation directly. Instead, we produce a new operad $O'$ whose algebras are left $O$-modules.

\begin{lemma}
Let $O$ be a one-color non-symmetric operad. Then left $O$-modules are algebras over $O'$ where $O'(n_1,\dots ,n_k; n)$ is $O_k$ if $n_1+\dots+n_k = n$, and empty otherwise.
\end{lemma}

\begin{proof}
One can check directly that the structure of an $O'$-algebra, i.e., morphisms of the form $O'(n_1,\dots,n_k; n) \otimes A_{n_1} \otimes \dots \otimes A_{n_k} \to A_n$ is the same as an $O$-algebra structure $O_k \otimes A_{n_1} \otimes \dots \otimes A_{n_k} \to A_n$.
\end{proof}

It is not hard to see that this construction can be generalized to $I$-colored non-symmetric operads. If $O$ is an $I$-colored non-symmetric operad, we can produce a new operad $O'$ whose algebras are left $O$-modules. The set of colors of $O'$ is the set of $I$-bouquets, that is the set of $k+1$-tuples $\vec{i} = (i_1,\ldots,i_k;i)$ in $I$ for varying $k \geq 0$. $O'\left(\vec{j}_1,\ldots,\vec{j}_k;\vec{j}\right)$ is given by $O(j_1,\ldots,j_k;j)$ if $\vec{j}_m = (j_{m1},\ldots,j_{ml_1};j_m)$ for $m=1,\ldots,k$ and $\vec{j} = (j_{11},\ldots,j_{ml_m};j)$. It is empty otherwise.

\subsection{Higher Extensions}\label{subsectionhigherextensions}
The cases of $Gr(Mon)$ and $Gr(NOp(I))$ fit into a general scheme. Let $T$ be a polynomial monad with the set of objects $I$ and let $T^+$ be its Baez-Dolan $+$-construction \cite{batanin-kock-joyal}. First, we recall from \cite{batanin-berger} that the classifier for $T^+$ is particularly easy to understand. For example, when $T$ is the monad for monoids, the classifier for $T^+$ can be again realized as linear trees, but now allowing circles for groupings of adjacent vertices. Each linear tree with circles can be unpacked into a planar tree as illustrated below. It follows that the classifier for $T^+$ is the same as for non-symmetric operads. The following picture illustrates:

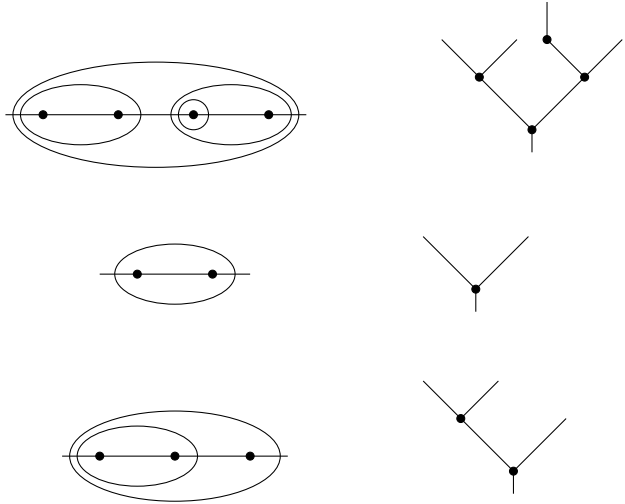
\[
\begin{tikzpicture}
\draw (0,0) -- (4,0);
\draw[fill] (.5,0) circle (1.5pt);
\draw[fill] (1.5,0) circle (1.5pt);
\draw[fill] (2.5,0) circle (1.5pt);
\draw[fill] (3.5,0) circle (1.5pt);
\draw (1,0) ellipse (.8 and .4);
\draw (3,0) ellipse (.8 and .4);
\draw (2,0) ellipse (1.9 and .7);
\draw (2.5,0) circle (.2);

\begin{scope}[shift={(7,-.2)}]
\draw (0,-.3) -- (0,0);
\draw (0,0) -- (-.7,.7);
\draw (0,0) -- (.7,.7);
\draw (-.7,.7) -- (-1.2,1.2);
\draw (-.7,.7) -- (-.2,1.2);
\draw (.7,.7) -- (.2,1.2);
\draw (.7,.7) -- (1.2,1.2);
\draw (.2,1.2) -- (.2,1.7);
\draw[fill] (0,0) circle (1.5pt);
\draw[fill] (-.7,.7) circle (1.5pt);
\draw[fill] (.7,.7) circle (1.5pt);
\draw[fill] (.2,1.2) circle (1.5pt);
\end{scope}
\end{tikzpicture}
\]

\[
\begin{tikzpicture}
\draw (0,0) -- (2,0);
\draw[fill] (.5,0) circle (1.5pt);
\draw[fill] (1.5,0) circle (1.5pt);
\draw (1,0) ellipse (.8 and .4);

\begin{scope}[shift={(5,-.2)}]
\draw (0,-.3) -- (0,0);
\draw (0,0) -- (-.7,.7);
\draw (0,0) -- (.7,.7);
\draw[fill] (0,0) circle (1.5pt);
\end{scope}
\end{tikzpicture}
\]

\[
\begin{tikzpicture}
\draw (0,0) -- (3,0);
\draw[fill] (.5,0) circle (1.5pt);
\draw[fill] (1.5,0) circle (1.5pt);
\draw[fill] (2.5,0) circle (1.5pt);
\draw (1,0) ellipse (.8 and .4);
\draw (1.5,0) ellipse (1.4 and .6);

\begin{scope}[shift={(6,-.2)}]
\draw (0,-.3) -- (0,0);
\draw (0,0) -- (-.7,.7);
\draw (0,0) -- (.7,.7);
\draw (-.7,.7) -- (-1.2,1.2);
\draw (-.7,.7) -- (-.2,1.2);
\draw[fill] (0,0) circle (1.5pt);
\draw[fill] (-.7,.7) circle (1.5pt);
\end{scope}
\end{tikzpicture}
\]

For a general polynomial monad $T$, given by $I \gets B \to E \to I$, the objects of $T^+$ are elements of $B$, and the operations of $T^+$ are trees on $B$, denoted $tr(B)$. Such a tree $q$ (with nodes in $B$), represents an operation taking the sources of $q$ (the input edges $s_1(q),\dots,s_n(q)$) to the target of $q$ (the root). 

A $T^+$-algebra $A$ is a collection of objects, one for each $b\in B$, such that, for every tree in $tr(B)$, there is a composition rule $A(s_1(q))\otimes \dots \otimes A(s_n(q))\to A(t(q))$.

Given a $T^+$-algebra $A$, we can define an $A$-algebra $C$ to be a $T$-algebra twisted by $A$, meaning there are objects $(C(i))_{i\in I}$ with $A(b)\otimes C(s_1(b))\otimes \dots \otimes C(s_k(b))\to C(t(b))$. Such objects can also be defined via the endomorphism $T^+$-algebra, as in \cite{batanin-berger}.

Then there is a canonical morphism of polynomial monads
$$\phi: T^+\to SOp(I)$$
which, according to Proposition \ref{prop:poly-for-Gr}, means there is a polynomial monad $Gr(T^+)$ for pairs $(P,A)$ where $P$ is an algebra of $T^+$ and $A$ is an algebra of $P$.

This morphism is constructed as follows. Recall that the polynomial for $T^+$ has a form

\begin{diagram}B&\lTo^{s}&tr^*(B)&\rTo^{p}&tr(B)&\rTo^{t}&
	B\end{diagram} 

where $tr(B)$ is the set of isomorphism classes of rooted trees decorated by elements from $B.$   
The morphism $\phi$ is represented by the following diagram (recall Proposition \ref{prop:poly-for-Gr}): 
\begin{diagram}B&\lTo^{s}&tr^*(B)&\rTo^{p}&tr(B)&\rTo^{t}&
	B\\
	\dTo^{c}&&\dTo&&\dTo_{\psi}&&\dTo_{c}\\
	Bq(I)&\lTo& ORTr^*(I) &\rTo & ORTr(I)&\rTo & Bq(I)\end{diagram}

Here $c$ sends an operation $b$ to its bouquet. The morphism $\psi$ sends a tree decorated by elements of $B$ to the same tree decorated by bouquets of elements.

\subsection{Further examples of the Grothendieck construction}

There are many ways to encode algebraic structure that fall under the banner of the Grothendieck construction. For example, to every club $T$ we can associate a category of $T$-algebras \cite{kelly}. The resulting Grothendieck construction has objects $(T,A)$ where $T$ is a club and $A$ is a $T$-algebra. We can do the same for Lawvere theories $L$, since every one has a category of $L$-modules. We can also do this for PRO(P)s. In general, categories of fibers in these situations are not known to have model structures, so the machinery of \cite{harpaz-prasma-integrated} cannot be used to produce global model structures. Unfortunately, our machinery cannot be used either, because these categories $\int \Phi$ are not algebras over a set-valued polynomial, though the category of clubs can be encoded via a category-valued polynomial, and a suitable extension of our machinery could be used to produce global (semi-)model structures in these settings, over a nicely behaved base model category. An easier example is the arrow category $\M^2$ of a model category $\M$. The codomain functor yields a Grothendieck construction \cite[Example 4.1]{cagne-mellies}, and this example is related to the theory of Smith ideals of ring spectra \cite{hovey-smith, white-yau-5, white-yau6}.

\section{Quasi-tame polynomial monads}\label{sec:quasi-tame}

In this section, we prove a general theorem that allows us to transfer model structures (and left properness) to categories of $T$-algebras, where $T$ is a quasi-tame polynomial monad, a generalization of the notion of tame from \cite{batanin-berger}. Motivated by applications to $n$-operads, we actually work in the context of quasi-tame substitudes. But we start with polynomial monads for simplicity. After setting up the filtration, we prove the transfer result. Then we prove that most of the monads from the previous section are quasi-tame, and this allows us to deduce the existence of the global model structure, consequences of which are explored in \cite{companion}.

We start this section by reviewing and generalizing the filtration of \cite{batanin-berger} for free-algebra extensions over polynomial monads. Much of our exposition, and many of our definitions, are drawn from \cite{batanin-berger}. The purpose of the filtration is to make the conditions of Theorems \ref{admissible} and \ref{adequate} easier to check. We begin with a reminder of classifiers, picking up where Section \ref{subsec:classifiers} left off.

\subsection{Classifier for free algebra extensions}\label{freealgext}  

Let $T$ be a finitary monad. In order to transfer a model structure to $T$-algebras, we must compute pushouts like (\ref{diagram:pushout}). The starting data is a span in the underlying category $\M$:
\begin{align} \label{diagram:pushout-in-M}
\xymatrix{
K \ar[r]^f \ar[d]_g & L \\
U(X) & }
\end{align}
where $X$ is a $T$-algebra. It is observed in \cite{batanin-berger} that to compute such a pushout one can construct a monad $T_{f,g}$ whose algebras are 5-tuples $(X,K,L,f,g)$ assembling to diagrams as above, together with a morphism of the monads $a: T_{f,g}\to T$ such that for any symmetric monoidal category $\M$, the left adjoint to the restriction functor
$$a_!:Alg_{T_{f,g}}(\M)\to Alg_T(\M)$$
is exactly the pushout $P$ of (\ref{diagram:pushout}). 

Moreover, if $T$ is a polynomial monad in $\Set$ the monad $T_{f,g}$ is polynomial and $a$ is a cartesian morphism of polynomial monads. Hence, we can apply the theory of classifiers and obtain $U(P)$ (coming from the pushout (\ref{diagram:pushout})) as a colimit over a classifier $T^{T_{f,g}}.$ 

It was observed that if $T$ is tame the classifier $T^{T_{f,g}}$ has a final subcategory $\mathbf{t}$ with a canonical filtration $\mathbf{t}^{(n)}$ such that $U(P)$ is a sequential colimit in $\M$ of objects $P_n$, which are themselves colimits over  $\mathbf{t}^{(n)}.$ Finally, $P_n$ itself is a pushout of the form  
\begin{align*}
\xymatrix{Q_n\ar[r]   \ar[d] \po & L_n\ar[d]\\ P_{n-1} \ar[r] & P_n}
\end{align*}
where $Q_n$ is a pushout over a punctured $n$-dimensional cube.
This allows us to study the homotopical properties of the pushout for $P$ using the monoid and pushout-product axioms. 

If $T$ is not tame one can not immediately apply the results of \cite{batanin-berger}, but we will see that in our case of interest most of the discussion from \cite{batanin-berger} goes through. The same observation applies even to some monads which are not polynomial in $\Set$ but are obtained as a canonical lifting of operads in $\Set$ as mentioned in Section \ref{sec:non-poly}.

Recall the polynomial monad $T+1$ (Section \ref{subsec:classifiers}) for pairs $(A,L)$, where $A$ is a $T$-algebra and $L \in \M$. The classifier $T^{T+1}$ encodes semi-free coproducts $T(L)\coprod A$, taken in $T$-algebras. There is a canonical morphism of monads $T+1 \to T$, whose induced restriction functor is given by $A \to (A,U(A))$. This semi-free coproduct is a special case of (\ref{diagram:pushout}) where $K$ is the initial object of $\M$. We define $T+2$ to be $(T+1)+1$.

There are morphisms of polynomial monads over $T$ 
$$T+1 \to T+2 \to T_g.$$ 
The first morphism is the identity on the $T$ summand of $T+1$  and sends the summand $1$ to the second $1$ of  $T+2$. The second morphism was described before. This induces a morphism of classifiers 
$$T^{T+1} \to T^{T+2} \to T^{T_g}.$$
It is not hard to see that explicitly $T^{T+1}$ can be realized as a full subcategory of $T^{T+2}$ which consists of objects with $X$ and $L$ edges only.  
 
Analogously, there is a morphism of monads  
 $$T+1 \to T+2 \to T_f$$ 
  and a morphism of classifiers
  $$T^{T+1} \to T^{T+2} \to T^{T_f}.$$

These morphisms are used to relate the classifiers $T^{T+1}, T^{T_f}, T^{T_g}$, and $T^{T_{f,g}}$. Having recalled the necessary notation from \cite{batanin-berger}, we are ready to give the substitude version.

\subsection{Classifiers for substitudes}

Let $(P,A)$ be a $\Sigma$-free substitude in $\Set$ (Definition \ref{defn:substitude}). Let ${\mathbb P}_{f,g}$ be the category whose objects are quintuples
$(X,K,L,g,f),$  where $X$ is a $P$-algebra, $K,L$  are objects in $[A,\Set]$ and $g:K\rightarrow \eta^*(X),\,f:K\rightarrow L$ are morphisms in $[A,\Set].$ There is an obvious forgetful functor
$$\Uu_{f,g}:{\mathbb P}_{f,g} \rightarrow [A_0,\Set]\times [A_0,\Set]\times [A_0,\Set],$$
taking the quintuple $(X,K,L,f,g)$ in $\mathbb{P}_{f,g}$ to the triple $(\eta_0^*(X),i^*(K),i^*(L)),$ 
where $i:A_0\to A$ is the inclusion of the maximal discrete subcategory of $A,$  and $\eta_0$ is the composite of the unit $\eta$ and $i$ as usual.  

\begin{pro}\label{Cart}

\begin{itemize} \item[(i)] The functor $\Uu_{f,g}$ is monadic and the induced monad $\Pp_{f,g}$ is polynomial;
\item[(ii)] There is a commutative square of adjunctions

\begin{equation*}\label{adjointext} \xygraph{!{0;(3.5,0):(0,.5)::}
{{\mathbb P}_{f,g}}="p0" [r] {\Alg_P}="p1" [d] {[A_0,\Set]}="p2" [l] {[A_0,\Set] \times [A_0,\Set]\times [A_0,\Set]}="p3"
"p0":@<-1ex>@{<-}"p1"_-{R_P}|-{}="cp":@<1ex>"p2"^-{\eta_0^*}|-{}="ut":@<1ex>"p3"^-{\Delta_{}}|-{}="c":@<-1ex>@{<-}"p0"_-\Uu_{f,g}|-{}="us"
"p0":@<1ex>"p1"^-{L_P}|-{}="dp":@<-1ex>@{<-}"p2"_-{(\eta_0)_!}|-{}="ft":@<-1ex>@{<-}"p3"_-{- \sqcup -}|-{}="d":@<1ex>"p0"^-\Ff_{f,g}|-{}="fs"
"dp":@{}"cp"|-{\perp} "d":@{}"c"|-{\perp} "fs":@{}"us"|-{\dashv} "ft":@{}"ut"|-{\dashv}}
\end{equation*}

\noindent in which  $\Delta$  is the diagonal and $R_P$ is given by $$R_P(Y)= (Y,\eta^*(Y),\eta^*(Y),id_{\eta^*(Y)},id_{\eta^*(Y)}).$$

\item[(iii)]The left adjoint $L_P$ to $R_P$ is given by the following pushout in $\Alg_P$:

\begin{align*} \label{free cofibration}
\xymatrix{
\eta_!(K) \ar[r]^{\eta_!(f)} \ar[d]_{\hat{g}} & \eta_!(L) \ar[d]^{} \\
X \ar[r]^{} & L_P(X,K,L,g,f)
}
\end{align*}
\noindent in which $\hat{g}$ is the mate of $g$.
\end{itemize}

\end{pro}

\begin{proof} The proof is completely analogous to the proof of Proposition 7.2 in \cite{batanin-berger}
\end{proof} 

We will need an explicit description of the classifier $\Pp^{\Pp_{f,g}}.$ It coincides almost verbatim to the description of $T^{T_{f,g}}$  given in \cite[Section 7.4]{batanin-berger}. So, we recall it briefly.

The objects of $\Pp^{\Pp_{f,g}}$ are corollas decorated by the elements of $B=\Pp(1)$ with its $|p^{-1}(b)|$ incoming edges colored by one of the three colors $X, K, L$:
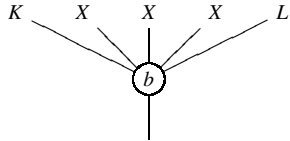
\begin{equation}\label{tricolor} \xygraph{!{0;(.7,0):(0,1)::} {\scriptstyle{b}} *\xycircle<6pt>{-} (-[l(2)u] {\scriptstyle{K}},-[ul] {\scriptstyle{X}},-[u] {\scriptstyle{X}},-[ru] {\scriptstyle{X}},-[r(2)u] {\scriptstyle{L}},-[d(.9)])} \end{equation}
These incoming edges will be called  $X$-edges, $K$-edges or $L$-edges accordingly.

Morphisms of $\Pp^{\Pp_{f,g}}$ can be described in terms of generators and relations. There are three types of generators. First, we have the generators coming from the $P$-algebra structure on $X$-colored edges and unary operations on $K$ and $L$ edges corresponding to morphisms of $A.$ The relations between these generators witness the relations between operations in $(P,A)$. We will call these generators \emph{$X$-generators}.

The next type of generators corresponds to the morphism $f:K\rightarrow L.$ Such a generator simply replaces a $K$-edge with an $L$-edge in the corolla. Generators of this kind will be called \emph{$F$-generators}. 

Finally, we have generators corresponding to $g:K\rightarrow \eta^*(X).$ Such a generator replaces a $K$-edge with an $X$-edge. Generators of this kind will be called \emph{$G$-generators}. The following lemma asserts how the generators commute and distribute over each other.

\begin{lem} \label{dist}  \begin{enumerate} \item Any span $b\stackrel{\phi}{\leftarrow} a\stackrel{\psi}{\rightarrow}a'$ in which $\phi$ is an $F$-generator (resp. $G$-generator) and  $\psi$ is an $X$-generator, completes uniquely to a commutative square
\begin{equation}\label{relation1}
\xygraph{{a}="p0" [r] {a'}="p1" [d] {b'}="p2" [l] {b}="p3" "p0":"p1"^-{\psi}:"p2"^-{\phi'}:@{<-}"p3"^-{\psi'}:@{<-}"p0"^-{\phi}} 
\end{equation}
in which $\phi'$ is an $F$-generator (resp. $G$-generator) and $\psi'$ is an $X$-generator. 

 \item Any cospan $b\stackrel{\psi'}{\to} b'\stackrel{\phi'}{\ot}a'$ in which $\phi'$ is an $F$-generator  and  $\psi'$ is an $X$-generator can be completed  to a commutative square
\begin{equation}\label{relation1}
\xygraph{{a}="p0" [r] {a'}="p1" [d] {b'}="p2" [l] {b}="p3" "p0":"p1"^-{\psi}:"p2"^-{\phi'}:@{<-}"p3"^-{\psi'}:@{<-}"p0"^-{\phi}} 
\end{equation}
in which $\phi$ is an $F$-generator and $\psi$ is an $X$-generator.

\item Any sequence of morphisms $a\stackrel{\phi}{\rightarrow} b\stackrel{\psi}{\rightarrow}b'$ in which $\phi$ is an $X$-generator and $\psi$ is an $F$-generator (resp. $G$-generator) can be completed to a commutative square 
\begin{equation}\label{relation 2}
\xygraph{{a}="p0" [r] {a'}="p1" [d] {b'}="p2" [l] {b}="p3" "p0":"p1"^-{\psi'}:"p2"^-{\phi'}:@{<-}"p3"^-{\psi}:@{<-}"p0"^-{\phi}} 
\end{equation}
in which $\phi'$ is an $X$-generator  and $\psi'$ is an $F$-generator (resp. $G$-generator). 
\item Any sequence of morphisms $a\stackrel{\phi}{\rightarrow} b\stackrel{\psi}{\rightarrow}b'$ in which $\phi$ is an $F$-generator and $\psi$ is an $X$-generator can be completed  to a commutative square
\begin{equation}\label{relation 2}
\xygraph{{a}="p0" [r] {a'}="p1" [d] {b'}="p2" [l] {b}="p3" "p0":"p1"^-{\psi'}:"p2"^-{\phi'}:@{<-}"p3"^-{\psi}:@{<-}"p0"^-{\phi}} 
\end{equation}
in which $\phi'$ is an $F$-generator  and $\psi'$ is an $X$-generator. 

\end{enumerate} 

\end{lem}

We now proceed like in Section 7 of \cite{batanin-berger} and introduce other monads associated to $(P,A).$ 

Let ${\mathbb P}_{f}$ be the category whose objects are quadruples
$(X,K,L,f),$ where $X$ is a $P$-algebra, $K,L$  are objects in $[A,\Set]$ and $f:K\rightarrow L$ is a morphism in $[A,\Set].$

Let ${\mathbb P}_{g}$ be the category whose objects are quadruples $(X,K,L,g),$  where $X$ is a $P$-algebra, $K,L$  are objects in $[A,\Set]$ and $g:K\to \eta^*(X)$ is a morphism in $[A,\Set]$.

The obvious forgetful functors $\Uu_f:{\mathbb P}_f\to [A_0,\Set]\times [A_0,\Set] \times [A_0,\Set]$ and $\Uu_g:{\mathbb P}_g\to [A_0,\Set]\times [A_0,\Set]\times [A_0,\Set]$ are monadic, yielding monads $\Pp_f$ and $\Pp_g$ for which there are propositions analogous to Proposition \ref{Cart}. We leave the details to the reader.

We put $\Pp+2\Aa=(\Pp+\Aa)+\Aa$. This is also a polynomial monad on $[A_0,\Set]\times [A_0,\Set]\times [A_0,\Set]$ as are $\Pp_{f,g}$, $\Pp_f$ and $\Pp_g$.

There is a commutative square of forgetful functors over $[A_0,\Set]\times [A_0,\Set]\times [A_0,\Set]$ 
\begin{equation} 
\xygraph{{\xybox{\xygraph{!{0;(3,0):(0,.5)::} {\mathbb P_{f,g}}="p0" [r] {\mathbb P_f}="p1" [d] {\Alg_P \times [A,\Set] \times [A,\Set]}="p2" [l] {\mathbb P_g}="p3" "p0":"p1"^-{}:"p2"^-{}:@{<-}"p3"^-{}:@{<-}"p0"^-{}}}}
 \end{equation} 
 and all four forgetful functors have left adjoints so that we get a commutative square
of monad morphisms going in the opposite direction and augmented over $P_0$ via cartesian natural transformations: \begin{equation} 
{\xybox{\xygraph{!{0;(1.5,0):(0,.6667)::} {\Pp_{f,g}}="p0" [r] {\Pp_f}="p1" [d] {\Pp + 2\Aa}="p2" [l] {\Pp_g}="p3" "p0":@{<-}"p1"^-{}:@{<-}"p2"^-{}:"p3"^-{}:"p0"^-{}}}}\end{equation}
We thus obtain a commutative square of categorical $P$-algebra morphisms of the corresponding classifiers, which enables us to analyze the category structure of ${{\Pp}^{\Pp_{f,g}}}$.
\begin{equation}\label{classifiersquare} {\xybox{\xygraph{!{0;(1.5,0):(0,.6667)::} {{\Pp}^{\Pp_{f,g}}}="p0" [r] {{\Pp}^{\Pp_f}}="p1" [d] {{\Pp}^{\Pp + 2\Aa}}="p2" [l] {{\Pp}^{\Pp_g}}="p3" "p0":@{<-}"p1"^-{}:@{<-}"p2"^-{}:"p3"^-{}:"p0"^-{}}}}}
\end{equation}
Finally, we have a morphism of monads over $\Pp:$
$$\Pp+{\Aa} \to \Pp+2{\Aa}$$
which is the identity on $\Pp$ and sends $\Aa$ to the second copy of $\Aa$ in $(\Pp+\Aa)+\Aa.$  Thus we have a morphism of classifiers:
$$\Pp^{\Pp+{\Aa}} \to \Pp^{\Pp+2{\Aa}}.$$

\begin{lem}\label{tfg}
\begin{enumerate}\item
The  classifiers ${{\Pp}^{\Pp_{f,g}}},{{\Pp}^{\Pp_f}},{{\Pp}^{\Pp_g}}, {{\Pp}^{\Pp + 2\Aa}} $ all have the same object-set, and the diagram of (\ref{classifiersquare}) identifies ${{\Pp}^{\Pp_f}},{{\Pp}^{\Pp_g}}$ with subcategories of ${{\Pp}^{\Pp_{f,g}}}$ which intersect in ${{\Pp}^{\Pp + 2\Aa}} $ and which generate ${{\Pp}^{\Pp_{f,g}}}$ as a category.
\item The composite functor $$\Pp^{\Pp+{\Aa}} \to \Pp^{\Pp+2{\Aa}} \to  \Pp^{\Pp_f}     $$ 
has a left adjoint $p$ (as functor between categories) such that the counit of the adjunction is the identity. Thus $\Pp^{\Pp+{\Aa}} $ is a reflective subcategory of $\Pp^{\Pp_f}.$
\item The composite functor $$\Pp^{\Pp+{\Aa}} \to \Pp^{\Pp+2{\Aa}} \to  \Pp^{\Pp_g}     $$ 
has a left adjoint $r$ (again as a functor) such that the counit of the adjunction is the identity. Thus $\Pp^{\Pp+{\Aa}} $ is a reflective subcategory of $\Pp^{\Pp_g}.$
\end{enumerate}
\end{lem}
\begin{proof} The proof of the first point of the lemma is completely analogous to the proof of \cite[Lemma 7.3]{batanin-berger}. 

 The reflection $p$ on an object from $\Pp^{\Pp_f}$  replaces all $K$-edges by $L$-edges and the unit is obtained by iterated application of $F$-generators. 

 The reflection $r:\Pp^{\Pp_g}\to \Pp^{\Pp+\Aa}$   takes an object of $\Pp^{\Pp_g}$ and replaces all $K$-edges by $X$-edges. 
The unit of this adjunction is generated by applying $G$-generators to all $K$-edges in the object. 
\end{proof}

\subsection{Canonical filtration}\label{canfilt}

In this section, we provide the filtration that will be at the heart of most of our transfer proofs. This section generalizes \cite{batanin-berger}, which only provided the filtration for tame polynomial monads.

To shorten the notation, we put $\ts = \h.$ We say that an object $a$ of $\ts$ is of type $(p,q)$ if $a$ contains exactly $p$ $K$-edges and $q$ $L$-edges, and we call $p+q$ the \emph{degree} of $a$. Let $k\ge 1.$ 

\begin{notation} \label{notation}
We define: 
\begin{itemize}
\item $\ts^{(k)}$ to be the full subcategory of $\ts$ spanned by all objects of degree $\leq k;$
\item $\LX^{(k)}$ to be a full subcategory of $\cop \subset \ho\subset \ts$  spanned by all objects of degree $\leq k;$
\item  $\wa^{(k)}$ to be the full subcategory of $\ts$ spanned by all objects of degree exactly $k$;
\item $\qa^{(k)}$ to be the full subcategory of $\wa^{(k)}$ spanned by all objects of type $(p,k-p)$ such that $p\ne 0;$ 
\item  $\la^{(k)}$ to be the full subcategory of $\wa^{(k)}$ spanned by all objects of type $(0,k);$ 
\item  $\overline{\qa}^{(k)}$ to be the full subcategory of $\ts^{(k)}$ spanned by the objects not contained in $\la^{(k)}$.
\end{itemize}
\end{notation}

Observe that $\wa^{(k)}$ and $\qa^{(k)}$ are also full subcategories of $ \Pp^{\Pp_f}.$ 

\begin{example}
In the case when $(P,A)$ encodes a Grothendieck construction $\int \Phi$, the objects of $\overline{\qa}^{(k)}$ are boxed planar trees where all leaves are boxed, and labeled by $X$, $K$, or $L$. Each connected component consists a fixed string of $X$, $L$, and $K$ obtained by moving around the tree clockwise and recording the order of vertices encountered, with no consecutive $X$'s, because all possible reduction rules have been applied.
\end{example}

\begin{proposition}\label{filtration}\label{SS}
For any functor $\mathbf{X}:\h \rightarrow \M$, the colimit of $\mathbf{X}$ is a sequential colimit of pushouts in $\M$.

More precisely, for $S_k = \colim_{\ts^{(k)}} \mathbf{X}|_{\ts^{(k)}}$, we get $$S=\colim_{\ts} \mathbf{X}\cong \colim_k S_k,$$ where the canonical morphism $S_{k-1}\to S_k$ is part of the following pushout square in $\M$
\begin{equation}\label{KLX}
\xygraph{!{0;(1.5,0):(0,.6667)::}
{Q_k}="p0" [r] {L_k}="p1" [d] {S_k}="p2" [l] {S_{k-1}}="p3"
"p0":"p1"^-{w_k}:"p2"^-{}:@{<-}"p3"^-{}:@{<-}"p0"^-{\alpha_k}
"p2" [u(.4)l(.3)] (-[d(.2)],-[r(.15)])}
\end{equation}
in which $Q_k$ (resp. $L_k$) is the colimit of the restriction of $\mathbf{X}$ to $\qa^{(k)}$ (resp. $\la^{(k)}$).
\end{proposition}

\begin{proof} 
First, observe that $\ts^{(k-1)}$ is a final subcategory of $\overline{\qa}^{(k)}.$   Let $I:\ts^{(k-1)}\subset \overline{\qa}^{(k)}$ be the  canonical inclusion and let $b\in  \overline{\qa}^{(k)}.$  

Observe  that the objects of $\overline{\qa}^{(k)}$ which are not in the image of this inclusion must have the type $(p,q) , p+q =k, p\ne 0$ that is they have at least one $K$-edge. So application of a $G$-generator (that is replacing a $K$-edge by an $X$-edge) to $b$ gives a morphism to an object in $\ts^{(k-1)}.$ So the comma-category $b/I$ is not empty.

Let now $a\stackrel{\phi}{\ot} b\stackrel{\psi}{\to} c$ be a span in $\overline{\qa}^{(k)}$    with $a,c$ being two objects of  $\ts^{(k-1)}$ and assume that $b$ is not in  $\ts^{(k-1)}.$ Then using $\phi$ and $\psi$ both have representations as composites of generators in which there is at least one $G$-generator.
Using Lemma  \ref{tfg}
one can assume that such a representation contains this generator on the first place so that the original span is the following composite 
$$a\stackrel{\phi'}{\ot}a'\stackrel{g_1}{\ot}  b\stackrel{g_2}{\to}c'\stackrel{\psi'}{\to} c$$
where $a',c' \in \ts^{(k-1)}.$ Now if $g_1 = g_2 =g$ (and hence $a'=c'$) the following diagram connects the objects $\phi$ and $\psi$ in $b/I:$ 

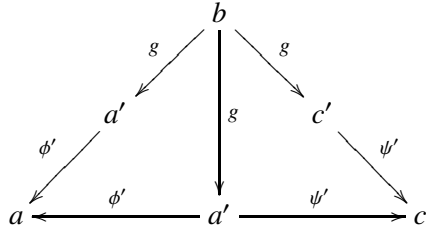
\begin{align}
\xymatrix{
&  & b \ar[ld]_g\ar[rd]^g \ar[dd]^g & &\\
&a'\ar[ld]_{\phi'} && c'\ar[rd]^{\psi'} & \\
a &  & a'  \ar[ll]_{\phi'} \ar[rr]^{\psi'} & & c
}
\end{align}

If $g_1\ne g_2$ it means that the generators $g_1,g_2$ replace different $K$-edges by $X$-edges. So, one can extend the span   $a'\stackrel{g_1}{\ot}  b\stackrel{g_2}{\to}c'$ to the commutative square by means of a cospan  $a'\stackrel{g'_2}{\to}  b'\stackrel{g'_1}{\ot}c'.$ All these data can be placed on the commutative diagram
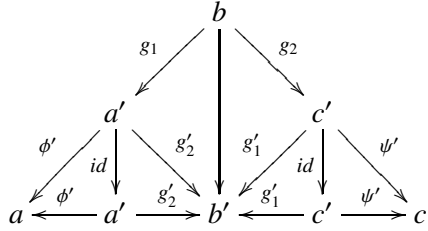
\begin{align}
\xymatrix{
&  & b \ar[ld]_{g_1}\ar[rd]^{g_2} \ar[dd] & &\\
&a'\ar[d]_{id} \ar[ld]_{\phi'}\ar[rd]^{g'_2} && c'\ar[rd]^{\psi'}\ar[ld]_{g'_1}\ar[d]_{id} & \\
a &  \ar[l]_{\phi'}a' \ar[r]^{g'_2} & b'  &  c' \ar[r]^{\psi'} \ar[l]_{g'_1}     & c
}
\end{align}
which again connect objects $\phi$ and $\psi.$
 
The category $\la^{(k)}$ is a reflective subcategory of $\wa^{(k)}$ by the restriction of the reflection $p$ from Lemma  \ref{tfg} (that is replacing all $K$-edges to $L$-edges).
     
Consider the following diagram of categories where the central square commutes: 
\begin{equation}\label{tamesquare} 
 \xygraph{!{0;(1.5,0):(0,.6667)::}
{{\qa}^{(k)}}="p0" [r] {\wa^{(k)}}="p1" [d] {\ts^{(k)}}="p2" [l] {\overline{\qa}^{(k)}}="p3"
"p0":"p1"^-{}:"p2"^-{}:@{<-}"p3"^-{}:@{<-}"p0"^-{}
"p3" :@{<-}[l] {\ts^{(k-1)}} "p1" :@{<-}[r] {\la^{(k)}}
} \end{equation}
Restricting $X$ to this subcategory and taking colimits we obtain a commutative diagram like (\ref{KLX}). We only need to know that this is a pushout diagram in $\M.$

A closer inspection of the central square (\ref{tamesquare}) reveals that it is a categorical pushout of a special kind: the category $\ts^{(k)}$ is obtained as the set-theoretical union of the categories $\overline{\qa}^{(k)}$ and $\wa^{(k)}$ along their common intersection $\qa^{(k)}$. Indeed, away from this intersection, there are no morphisms in $\ts^{(k)}$ between objects of $\overline{\qa}^{(k)}$ and objects of $\wa^{(k)}$. By Lemma 7.13 from \cite{batanin-berger}, this implies that (\ref{KLX}) is a pushout square in $\M$.

\end{proof}

\begin{example}
In the case where $(P,A)$ encodes monoids, the filtration of Proposition \ref{filtration} is especially nice, because $\mathbf{p}$ has a discrete final subcategory $\mathbf{t}$, described in Section \ref{subsec:monoids-modules}, and we can define the subcategories of Notation \ref{notation} relative to this $\mathbf{t}$. Set $Q_k = \colim_{\qa^{(k)}} \mathbf{X}$. Let $\qa^{(k)}(0)$ denote the subcategory of $\qa^{(k)}$ spanned by the objects with no $X$'s. Let $Q_k(0) = \colim_{\qa}^{(k)}(0)$. It is easy to see that $Q_k(0)$ is the domain of the iterated pushout product $f^{\boxprod k}$, i.e., is the colimit of the $n$-dimensional punctured cube with vertices words of length $n$ formed from the letters $K$ and $L$. For the intersection of $\qa^{(k)}$ with $\mathbf{t}$, there are also $X$'s between the $K$'s and $L$'s, forming an alternating string, like in Section \ref{subsec:monoids-modules}. The morphism $w_k$ in Proposition \ref{filtration} is precisely the morphism $X^{\otimes (k+1)}\otimes f^{\boxprod k}$ from \cite{SS00}.
\end{example}

\subsection{Model structure for internal algebras}

In this section, we use the filtration of Proposition \ref{filtration} to prove the existence of model structures on algebras over well-behaved substitudes. We make use of the model category terminology of Section \ref{subsec:transfer}

Let $(P,A)$ be a $\Sigma$-free substitude. Let $\Ee$ be a cocomplete categorical pseudoalgebra of $P.$ 
Then the monad morphism  $\eta:A\to P$ induces an adjunction between internal algebras of $A$ in $\Ee$ and 
$P$ in $\Ee.$ We call the category of internal $A$-algebras in $\Ee$ the category of internal $A$-collections $Coll_A(\Ee)$ and the category of internal $P$-algebras in $\Ee$ simply the category of $P$-algebras $\Alg_P(\Ee).$ 

\begin{remark} To clarify the situation, an $A$-collection $X$ is an internal $A$-presheaf in the categorical $A$-presheaves $\eta^*(\Ee).$ That is for each $a\in A,$ $X(a)\in \Ee(a)$ and for each $f:a\to b$ in $A$ we have a morphism $\Ee(f)(a)\to b$  in $\Ee(b)$ (thinking about $f$ as a unary  operation of $P$).  These morphisms satisfy an obvious coherence relation. For example, if $\Ee=\M^\bullet$ is a constant $P$-algebra, then an internal $A$-presheaf in $\Ee$ is just a presheaf in $\M.$ The main applications of our results concerns this classical constant case so the reader may keep in mind that $Coll_A(\Ee)$ is simply the category of presheaves. The argument in the general situation does not differ much. 
 \end{remark}

The unit $\eta$ of the substitude induces an adjunction between the category of  internal $A$-collections  and internal $P$-algebras  in $\Ee$ and this adjunction is monadic (cf. \cite{Reedy-paper} ). Assume now that  that  $Coll_A(\Ee)$ is equipped with a model structure. We can ask: when can this model structure can be transferred to the category of internal $P$-algebras?

Now let $X=(X,K,L,g,f) $ be an internal algebra of $\Pp_{f,g}$ in $\Ee$ and let $\widetilde{\XX}:\h \rightarrow Alg_P(\Ee)$ be a strict functor which represents $X.$ 
By restricting along $\eta$ we obtain a composite functor:
$$\XX:\h \rightarrow Alg_P(\Ee)\stackrel{\eta^*}{\to} Coll_A(\Ee).$$
On an object 
  $b\in \h$ (that is, a decorated corolla)
 \begin{equation}\label{XXX} \XX(b) = \bigotimes_{i\in \text{sources of~} X\text{-edges} }X_i \bigotimes_{j\in \text{sources of~} K\text{-edges}} K_j \bigotimes_{k\in \text{sources of~} L\text{-edges}} L_k  \end{equation}
 For example $\XX(b) = X\otimes X\otimes X \otimes K \otimes L$ on the corolla below
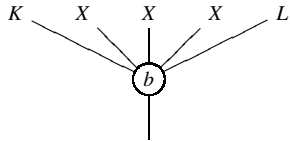
\begin{equation}\label{tricolor} \xygraph{!{0;(.7,0):(0,1)::} {\scriptstyle{b}} *\xycircle<6pt>{-} (-[l(2)u] {\scriptstyle{K}},-[ul] {\scriptstyle{X}},-[u] {\scriptstyle{X}},-[ru] {\scriptstyle{X}},-[r(2)u] {\scriptstyle{L}},-[d(.9)])} \end{equation}
In this example the monad has only one color. 

We are now in the situation of Proposition \ref{filtration} for $\M=Coll_A(\Ee).$
 
Let $\psi:T\to S$ be a morphism of  polynomial monads. Let $\Ee$ be a cocomplete categorical pseudoalgebra of $S.$ 
Then the monad morphism  $Id\to T$ induces an adjunction between internal algebras of $Id$ in $\Ee$ and 
$T$ in $\Ee.$ We call the category of internal $Id$-algebras in $\Ee$ the category of internal $\Ee$-collections $Coll(\Ee)$ and the category of internal $T$-algebras in $\Ee$ simply the category of $T$-algebras $\Alg_T(\Ee).$  
 
We suppose that the adjunction between the category of  internal collections  and internal $T$-algebras  in $\Ee$ is monadic and that there is a model structure on the category of $Coll(\Ee).$

\begin{proposition}\label{Ktheorem}  
Let $\Kk$  be a saturated class in $Coll_A(\Ee)$ such that $Coll_A(\Ee)$ is $\Kk$-compactly generated model category. And let $\Ll \subset \cat{W}\cap \Kk$ as usual. 
Then  the monad on $Coll_A(\Ee)$ induced by the adjunction between the category of  internal collections  and internal $T$-algebras  in $\Ee$  is
\begin{itemize}
\item[(a)] $(\Kk,\Ll)$-admissible if for any cofibration $u:K\to L$ in $Coll_A(\Ee)$ the morphisms
$$w_k(u): Q_k\to L_k$$
belong to  $\Kk$ and if $u$ is a trivial cofibration these morphisms belongs to $\Ll;$  
\item[(b)] relatively $(\Kk,\Ll)$-adequate if in addition $Coll_A(\Ee)$ is left proper, for any cofibration $u$ with cofibrant domain the morphism $w_k(u)$  is an $h$-cofibration and for any weak equivalence $\phi:X\to Y$ between relatively cofibrant $T$-algebras the 
morphisms $Q_k(\phi):Q_k(X)\to Q_k(Y)$ and $L_k(\phi): L_k(X)\to L_k(Y)$ are weak equivalences;
\item[(c)] $(\Kk,\Ll)$-adequate if  $Coll_A(\Ee)$ is left proper, if for any cofibration $u$ the morphism $w_k(u)$  is an $h$-cofibration, and if for any weak equivalence between internal $T$-algebras $\phi:X\to Y$ the morphisms $Q_k(\phi)$ and $L_k(\phi)$ are weak equivalences.
\end{itemize}

\end{proposition}

\begin{proof}

To simplify the notation, we will denote $Coll_A(\Ee)$ simply $Coll$ and the category of internal $P$-algebras in $\Ee$ simply $\Alg_P.$

To prove (a) we need to check that  
for each   (trivial) cofibration $u:K\to L$ and each morphism of $P$-algebras $\alpha:F_T(K)\to X$, the pushout in $\Alg_P$
\begin{equation}\label{extension} 
\xygraph{!{0;(2,0):(0,.5)::} 
{F_T(K)}="p0" [r] {X}="p1" [d] {X[u,\alpha]}="p2" [l] {F_T(L)}="p3" 
"p0":"p1"^-{\alpha}:"p2"^-{u_{\alpha}}:@{<-}"p3"^-{}:@{<-}"p0"^-{F_T(u)} 
"p2" [u(.4)l(.3)] (-[d(.2)],-[r(.15)]) 
}
\end{equation}
yields a $P$-algebra morphism $u_{\alpha}: X\to X[u,\alpha]$ whose underlying morphism $U_P(u_{\alpha})$ belongs  to $\Kk$ (resp. to $\Ll$).

According to Proposition \ref{SS} such a pushout can be computed as a sequential colimit, where each morphism is a certain pushout of $w_k(u).$  Since $\Kk$ and $\Ll$ are  saturated classes, the conclusion follows. 

For (b) observe that if $u:K\to L$ is a cofibration with cofibrant domain and $f:X\to Y$ is  a weak equivalence  between relatively cofibrant algebras we have the following commutative cube for each $k\ge 1$ 
\begin{equation}\label{cube2}
\xygraph{
*!(0,-2)=(0,3.5)
{\xybox{\xygraph{!{0;(2.5,0):(0,.75)::}
{Q_k(X)}="f0" [r] {P_{k-1}(X)}="f1" [d] {P_k(X)}="f2" [l] {L_k(X)}="f3"
"f0":"f1"^-{}:"f2"^-{}:@{<-}"f3"^-{}:@{<-}"f0"^-{}
"f0" [u(.5)r(.5)] {Q_k(Y)}="b0" [r] {P_{k-1}(Y)}="b1" [d] {P_k(Y)}="b2" [l] {L_k(Y)}="b3"
"b0":"b1"^-{}:"b2"^-{}:@{<-}"b3"|(.5)*=<5pt>{}:@{<-}"b0"|(.5)*=<5pt>{}
"f0":"b0" "f1":"b1" "f2":"b2" "f3":"b3"}}}}
\end{equation}
Now, the morphism $U(\phi) = \phi_0: U(X)=P_0(X)\to U(Y) = P_0(Y)$  is a weak equivalence. Let us prove by induction that  any $\phi_k$ is a weak equivalence, so the adequateness of the monad will follow from the fact that the category of collections is $\Kk$-compactly generated. 

Indeed, in the cube (\ref{cube2}) the arrow $P_{k-1}(X)\to P_{k-1}(Y)$ is a weak equivalence by the inductive assumption, and $Q_k(X)\to Q_k(Y), L_k(X)\to L_k(Y)$ are weak equivalences. It remains to show that both the front and back squares are homotopy pushouts. This follows from the condition that $Coll_A(\Ee)$ is left proper and $w_k(u)$ are $h$-cofibration, hence, the pushouts in front and back of the cube are homotopy pushouts.

The proof for statement (c) is similar to (b). 

\end{proof}

\begin{theorem}\label{Kcorol} 
Under the conditions of Proposition \ref{Ktheorem} 
\begin{itemize}
\item[(a)] The category of internal $P$-algebras in $\Ee$  admits a transferred model structure;
\item[(b)] This model structure is relatively left proper;
\item[(c)] This model structure is left proper.
\end{itemize}
\end{theorem}

\subsection{Quasi-tame substitudes} 

In \cite{batanin-berger}, many examples of polynomial monads are given, and we have seen more in Section \ref{sec:Gr(T)}. However, tameness is somewhat rare in general. For instance, the monads for reduced operads and for non-symmetric operads are tame, but the monads for their Grothendieck constructions are not tame.  

Our key new contribution to handle such situations is the notion of a quasi-tame polynomial monad. Such monads are much more common than tame polynomial monads, and Theorem \ref{thm:quasi-tame-admissible} can be applied to them in order to produce (relatively) left proper transferred model structures. If $T$ is tame, then the pushouts $S_k$ in the sequential colimit from Proposition \ref{filtration} are particularly easy to handle. However, quasi-tameness is sufficient for the arguments of Theorem \ref{thm:tame-admissible}, as we now show.

In this section, we work in the even more general setting of quasi-tame substitudes, which are quasi-tame polynomial monads in case $A$ is a discrete category (i.e., a set). We do this in order to get applications to $n$-operads. We then prove that algebras over a quasi-tame substitude admit a transferred model structure, which is furthermore left proper if the ambient model category $\M$ is sufficiently nice. The following is a generalization of the definition of tame polynomial monads, as we will prove below.

\begin{defin} A $\Sigma$-free substitude $(P,A)$ is called \textit{quasi-tame} if the fundamental groupoid  $\Pi_1(\cop)$ is equivalent to a discrete groupoid.
\end{defin}

This definition is equivalent to the statement that, on each connected component of $\cop$, the fundamental groupoid is trivial. The key reason why quasi-tameness is enough for a transferred model structure is that the colimit over $\h$ can be split into two colimits -- one with information about the $K$ and $L$ vertices, and the other with information about the $X$ vertices. For quasi-tame polynomial monads, the larger colimit can be computed from the smaller colimits. The proof of this statement requires a finality argument, which we will explain below.

Now we assume that  $\Ee = \M^\bullet$ for a cocomplete symmetric monoidal $\M$ so  that $Coll_A(\Ee) = [A,\M].$ The category $[A,\M]$ is symmetric monoidal category with respect to the pointwise tensor product $\otimes_p.$ When we spoke of operads, we wrote $\M^\Sigmac$ instead of $[\Sigmac,\M]$ because this is the standard notation in operad theory. However, when we speak of substitutes (which, recall, are category-colored operads) we instead write $[A,\M]$ to emphasize the role of the category $A$.

We further  assume that $[A,\M]$  is equipped with  a cofibrantly generated model category structure with the class of weak equivalences $\WW,$ and set of generating cofibrations $\Ii$ and set of generating trivial cofibrations $\Jj.$  

The following definitions are not new and were introduced in \cite{batanin-berger}. Here we specialize it to the pointwise monoidal structure on $[A,\M]$ by adding the adjective pointwise everywhere. This is done just for convenience in order to remember which structure we have in mind.     

\begin{defin} A model structure on $[A,\M]$ is called pointwise compactly generated if it is $\Ii^{\otimes_p}$-compactly generated where  
$$\Ii^{\otimes_p} = Sat(\{\Ii\otimes_p x, \ x\in [A,\M] \}).$$ 
\end{defin} 

\begin{defin} A model structure on $[A,\M]$ is called  pointwise monoidal  if   the pushout-product axiom  is satisfied for the pointwise monoidal structure on  $[A,\M].$\end{defin} 

\begin{defin}  The model structure on $[A,\M]$ satisfies a pointwise monoid axiom if the following containment holds, where $Sat$ means saturation:  
$$\Jj^{\otimes_p} = Sat(\{\Jj\otimes_p x, \  x\in [A,\M] \})\subset \WW.$$ 
\end{defin}

\begin{defin} A model structure on   $[A,\M]$ will be called pointwise $h$-monoidal if for any cofibration $f:x\to y$ and any preasheaf $z$  the morphism $f\otimes_p z$ is an $h$-cofibration , which is trivial if $f$ is trivial. It will be called strongly pointwise $h$-monoidal if weak equivalences are closed under pointwise tensor product.  
\end{defin}

As in the case of $h$-monoidal categories any cofibration in a pointwise $h$-monoidal category is an $h$-cofibration so it is left proper \cite{batanin-berger}. In addition the following analogue of Proposition 2.5 from \cite{batanin-berger} holds.
\begin{proposition} In a pointwise compactly generated h-monoidal category the pointwise monoid axiom holds. 
\end{proposition}

\begin{theorem} \label{thm:quasi-tame-admissible}
Let $(P,A)$ be a quasi-tame $\Sigma$-free substitude.  
\begin{enumerate}
\item If $[A,\M]$ satisfies the pointwise monoid axiom then the category $Alg_P(\M)$ admits a transferred model structure along $\eta^*:Alg_P(\M)\to [A,\M].$

\item If $[A,\M]$ is pointwise $h$-monoidal then this model structure on $Alg_P(\M)$ is relatively left proper.  

\item If $[A,\M]$ is strongly pointwise $h$-monoidal then this model structure on   $Alg_P(\M)$ is left proper.

\end{enumerate} 

\end{theorem}
\begin{proof}

We have to prove that for any cofibration $f:K\to L$ the canonical morphism $w_k:Q_k\to L_k$ is in $\Ii^{\otimes_p}$ and it is in $\WW\cap\Ii^{\otimes_p}$ if $f$ is a trivial cofibration.

 Let $X=(X,K,L,g,f) $ be an algebra of $\Pp_{f,g}$ and let ${\XX}:\h \rightarrow [A,\M]$ be the functor (\ref{XXX}).  To compute $Q_k$ it is enough to restrict $\XX$ to  $ \Pp^{\Pp_f}$ because  $\wa^{(k)}$ and $\qa^{(k)}$ are  full subcategories of $ \Pp^{\Pp_f}.$ This restriction represents the $\Pp_f$-algebra $(X,K,L,f)$ and we will denote it by $\XX$ also. 

Observe that the functor $\XX$ is canonically isomorphic to the product of two functors \begin{equation}\label{XxXf} \XX_x(b)\otimes_p \XX_f(b)\end{equation}  which are defined as follows:
$$\XX_x(b) = \bigotimes_{i\in \text{sources of~} X\text{-edges} }X_i  $$
and on $X$-generators it acts as $\XX$ (so one can think about it as replacing all $L$ and $K$ by the tensor unit $I$) but on $F$-generators it acts as  identity. 
The functor $\XX_f(b)$ is defined as
$$\XX_f(b) = \bigotimes_{j\in \text{sources of~} K\text{-edges}} K_j \bigotimes_{k\in \text{sources of~} L\text{-edges}} L_k  $$
and it acts as identity on $X$-generators and as $\XX$ on $F$-generators.  

We summarize the situation as the following commutative diagram
\begin{align} \label{decomposition}
\xymatrix@C = +4em@R = +2em{
 \Pp^{\Pp_{f}} \ar[r]^{\XX} \ar[d]_{\Delta} & [A, \M] \\
 \Pp^{\Pp_{f}}\times  \Pp^{\Pp_{f}} \ar[r]^{(\XX_x,\XX_f)} & [A,\M]\times [A,\M] \ar[u]_{-\otimes_p -} 
}
\end{align}

By the second point of Lemma \ref{tfg} we have a canonical isomorphism $\pi_0(\cop)\to \pi_0( \Pp^{\Pp_{f}} )$ which induces the following decomposition 
$$\qa^{(k)} \cong \coprod_{c\in \pi_0(\cop)} \qa^{(k)}_c. $$  
Thus the colimit of a functor over $\qa^{(k)}$ is a coproduct of the colimits over the categories $\qa^{(k)}_c.$  We now focus on how we can compute these colimits. 

Restricting diagram (\ref{decomposition}) to $\qa^{(k)}_c$ we get  a diagram: 
\begin{align} \label{diagram:quasi-tame finality triangle}
\xymatrix@C = +4em@R = +2em{
\qa_c^{(k)}\ar[r]^{\XX} \ar[d]_{\Delta} & [A,\M] \\
\qa^{(k)}_c\times  \qa^{(k)}_c \ar[r]^{(\XX_x,\XX_f)} &[A, \M]\times [A,\M] \ar[u]_{-\otimes_p -}  
}
\end{align}

Let now $\qa^{(k)}_c[X^{-1}]$ be the localisation of $\qa^{(k)}_c$ at the subset of $X$-generators and similarly $\qa^{(k)}_c[F^{-1}]$ be the localization of $\qa_c^{(k)}$ at the subset of $F$-generators and observe that $\XX_x$ can be factorized through $\qa_c^{(k)}[F^{-1}]$ and $\XX_f$ can be factorized through $\qa_c^{(k)}[X^{-1}].$
We then have a commutative diagram of functors
\begin{align} 
\xymatrix@C = +4em@R = +2em{
\qa_c^{(k)} \ar@/^-7.3ex/[dd]_{\db} \ar[r]^{\XX} \ar[d]_{\Delta} & [A,\M]  \\
\qa_c^{(k)}\times  \qa_c^{(k)}\ar[d] \ar[r]^{(\XX_x,\XX_f)} & [A,\M]\times [A,\M] \ar[u]_{-\otimes_p -} \\
\qa_c^{(k)}[F^{-1}]\times \qa_c^{(k)}[X^{-1}]\ar[ru]_{(\XX_x,\XX_f)}& 
}
\end{align}
We are going to prove that the left vertical composite  $\db$ is a final functor, so that the colimit of $\XX$ over $\qa_c^{(k)}$ is isomorphic to the  colimit of 
$\XX_x\otimes \XX_f$ over the category $\qa_c^{(k)}[F^{-1}]\times \qa_c^{(k)}[X^{-1}].$

\begin{lem}\label{PAA}  For a $\Sigma$-free quasi-tame substitude $(P,A)$ the   fundamental groupoid  $\Pi_1(\Pp^{\Pp+2{\Aa}})$ is equivalent to a discrete groupoid.

\end{lem}

\begin{proof}  Observe the classifier  (as a category) $\Pp^{\Pp+2{\Aa}}$ can be decomposed as a coproduct of categories 
$$\Pp^{\Pp+2{\Aa}} \simeq \coprod_{k\ge 0} \coprod_{p+q = k} \Pp^{\Pp+2{\Aa}}(p,q),$$
where $\Pp^{\Pp+2{\Aa}}(p,q)\subset \Pp^{\Pp+2{\Aa}}$ is the full subcategory of objects of type $(p,q).$ Since we have no $F$-generators in the classifier $\Pp^{\Pp+2{\Aa}}$ we have the following isomorphisms of categories for every $(p,q)$ with $p+q = k.$   $$ \Pp^{\Pp+2{\Aa}}(p,q) \simeq \Pp^{\Pp+2{\Aa}}(k,0) = \cop(k)$$
where $\cop(k)\subset \cop$ is the full subcategory of objects of type $(k,0).$
We finish the proof by observing that 
   $$\cop \simeq \coprod_{k\ge 0} \cop(k).$$

\end{proof}

We can now show that the comma-category $(b,b')/\db$ is nonempty and connected for every $(b, b')$ in $\qa_c^{(k)}[F^{-1}]\times \qa_c^{(k)}[X^{-1}].$

For nonemptiness we must find $(b'',b'')$ in the image of $\db$ with a morphism from $(b ,b')$ to $(b'',b'')$ in $\qa_c^{(k)}[F^{-1}]\times \qa_c^{(k)}[X^{-1}]$. 
By assumption, the $b$ and $b'$ live in the same connected component      $\qa_c^{(k)}$. Thus we have a zig-zag of morphisms in   $\qa_c^{(k)}$ where each morphism
is an $X$ or $F$-generator: 
$$b\ot b_0 \to b_1 \ot b_2 \to \cdots \ot b_n \to b'.$$ 
Using distributivity and commutativity of $X$-generators over $F$-generators from Lemma \ref{dist} we can assume that in this zig-zag there is a $b_k = b''$  such that on the left-hand side of $b''$ there are only $F$-generators but on the right-hand side there are only $X$-generators in the sequence. Then the left-hand side represents a morphism $b\to b''$ in  $\qa_c^{(k)}[F^{-1}]$ and the right-hand side represents a morphism $b'\to b''$ in $\qa_c^{(k)}[X^{-1}]$ so together they provide a morphism $(b ,b') \to (b'',b'')$ in $\qa_c^{(k)}[F^{-1}]\times \qa_c^{(k)}[X^{-1}]$.

Next, we study morphisms in the categories  $\qa_c^{(k)}[F^{-1}]$ and $ \qa_c^{(k)}[X^{-1}]$. First observe that any cospan of $F$-generators $b\stackrel{f_1}{\to} b'\stackrel{f_2}{\ot} b''$ in $\qa_c^{(k)}$ can be completed to a commutative square 

\begin{equation}
\xygraph{{b'''}="p0" [r] {b''}="p1" [d] {b'}="p2" [l] {b}="p3" "p0":"p1"^-{f_1'}:"p2"^-{f_2}:@{<-}"p3"^-{f_1}:@{<-}"p0"^-{f_2'}} 
\end{equation}
in which $f_1',f'_2$ are $F$-generators. To find such a $b'''$ it is enough to change all $L$-edges in $b$ to $K$-edges.  

It follows from this and Lemma \ref{dist} that a typical morphism in  $\qa_c^{(k)}[F^{-1}]$ from $b\to b''$ can be represented by a zigzag of morphisms of the form
$$b\stackrel{f_1}{\ot} b_0 \stackrel{f_2}{\to} b_1\stackrel{x}{\to} b''$$
where $f_1,f_2$ are $F$-generators and $x$ is an $X$-generator.

Similarly a morphism from $b'\to b''$ in $\qa_c^{(k)}[X^{-1}]$ can be represented by a zigzag of the form:
$$b'\stackrel{x_0}{\ot} b_0 \stackrel{x_1}{\to} b_1 \ot \cdots \to b_{n-1} \stackrel{x_n}{\to}b_n \stackrel{f}{\to}b''$$ 
where $x_0,\ldots,x_n$ are $X$-generators and $f$ is an $F$-generator. 

An object $(b,b')\to (b'',b'')$ in $(b,b')/\db$ is, therefore, represented by  a zigzag of the form
\begin{equation} \label{zz1} b\stackrel{f_1}{\ot} b_0 \stackrel{f_2}{\to} b_1\stackrel{x}{\to} b'' \stackrel{f}{\ot} b_n \stackrel{x_n}{\to} b_{n-1} \ot \cdots \to b_{1} \stackrel{x_1}{\ot}b_0 \stackrel{x_0}{\to} b' .\end{equation}
We now have to show that given a similar zigzag $(b,b')\to (b''',b''')$ 
\begin{equation}\label{zz2} b\stackrel{f'_1}{\ot} b'_0 \stackrel{f'_2}{\to} b'_1\stackrel{x'}{\to} b''' \stackrel{f'}{\ot} b'_{n'} \stackrel{x'_{n'}}{\to} b'_{n'-1} \ot \cdots \to b'_{1} \stackrel{x'_1}{\ot}b'_0 \stackrel{x'_0}{\to} b' \end{equation}
one can connect $b''$ and $b'''$ by a zigzag of morphisms in $\qa_c^{(k)}$ together with corresponding zigzags making everything commute. 

For this we first pay attention to the cospan 
$$b_1\stackrel{x}{\to} b'' \stackrel{f}{\ot} b_n$$ in the middle of the zigzag (\ref{zz1}). By Lemma \ref{dist} one can replace it by a span 
$$b_1\stackrel{\bar{f}}{\ot} \bar{b''} \stackrel{\bar{x}}{\to} b_n.$$
Moreover, one can assume that $\bar{b''}$ does not have $L$-colored edges because if it does we can always cover it by an object without any. 
We now can see that we have a morphism in  $(b,b')/\db$ from  the object $(b,b')\to (b,\bar{b''})$  (a zigzag obtained by replacement of the middle cospan by a span)  to  the object $(b,b')\to (b'',b'').$  Observe that all morphisms in the zigzag $(b,b')\to (b,\bar{b''})$ from the left hand side of $\bar{b''}$ are $F$-generators but all morphisms from the right hand side are $X$-generators. 

We do a similar transformation for the zigzag (\ref{zz2}) and observe that the new object $\bar{b'''}$ is related to $\bar{b''}$ by a zigzag of $F$-generators and both do not have $L$-colored edges. Hence, $\bar{b''} = \bar{b'''} = m$ and we have the following  diagram in $\qa_c^{(k)}:$ 
\begin{align} 
\xymatrix@C = +3em@R = +1em{
 & & b_1\ar[r]^x & b'' & \ar[l]_f b_n & &  \\
 & \ar[dl] b_0 \ar[ur]& &   & & &  \\
  b    &  &  &\ar[lll] d \ar[dd] \ar[uu] \ar[ruu]_{\bar{x}}\ar[rdd] \ar[ldd]\ar[lld] \ar[llu]\ar[luu]^{\bar{f}}   & & & \ar@{-->}[lluu]  \ar@{-->}[lldd] b'\\
  & \ar[lu] b'_0 \ar[rd]& & & & & \\
  & & b'_1\ar[r]&  b''' &\ar[l]  b'_{n'} & &      
}
\end{align}
in which the subdiagram of solid arrows commutes and in which the right-hand side diagram consists of two solid arrows and dashed arrows representing zigzags of $X$-generators in $\qa_c^{(k)}$. Hence, by virtue of Lemma \ref{PAA} it can be filled in to the commutative diagram using only $X$-generators. Hence the finality of $\db$ is proved.   

\medskip

We now come back to the proof that $w_k:Q_k\to L_k$ is in $\Ii^{\otimes_p}.$ It follows from what we have just proved  that the colimit of $\XX$ over $\qa_c^{(k)}$ is canonically isomorphic to the colimit of $\XX_x\otimes_p \XX_f$ over  $\qa_c^{(k)}[F^{-1}]\times \qa_c^{(k)}[X^{-1}].$ Then we observe that $w_k$ is the usual comparison morphism in the punctured $k$-cube and so the standard argument \cite{batanin-berger,SS00} show that under the pointwise pushout-product and pointwise monoid axioms for $\otimes_p$ it is in $\Ii^{\otimes_p}$ if $f$ is a cofibration and it is a weak equivalence if $f$ is a trivial cofibration. The proof of the rest of the theorem also follows standard arguments from \cite{batanin-berger}.   
\end{proof}

We conclude with a multi-paragraph remark illustrating what goes wrong for monads that fail to be quasi-tame.

\begin{remark}
Computations show that Theorem \ref{thm:quasi-tame-admissible} is close to a characterization of quasi-tameness, in the sense that, if $(P,A)$ is a $\Sigma$-free substitude that is admissible for all $\M$ such that $[A,\M]$ satisfies the pointwise monoid axiom, then $(P,A)$ is quasi-tame. Indeed, let $\M$ be the projective model structure on the category of chain complexes over a field of characteristic $p>0$. Let $(P,A)$ be a non-quasi-tame substitude, meaning that $\Pi_1(\cop)$ is not equivalent to a discrete groupoid. We conjecture that there cannot be a transferred model structure in this setting. A possible line of attack is to convert the non-quasi-tameness of $(P,A)$ into a symmetric group action, with which we can mimic Example 2.9 of \cite{batanin-white-eilenberg-moore} (where we showed that the substitude for $I$-colored symmetric operads is not admissible).

Skipping many technical details, we first take a coproduct of the terminal algebra of $P$ and the free algebra on a cofibrant contractible collection $K_i, i\in I$. Then $\Pi_1(\cop)$ acts on certain tensor products of $K_i$ by permutations, and we will show below that all morphisms in $\Pi_1(\cop)$ are just these permutations. More precisely, this action exhibits an embedding of the groupoid $\Pi_1(\cop)$ into the groupoid of permutations on the $K_i$. The reason is that morphisms in $\Pi_1(\cop)$ are generated by compositions in $P$, their inverses, and certain permutations of the $K_i$ which might happen when we compose operations in $P$. A nontrivial loop can only happen as a nontrivial permutation of $K_i$ terms. Indeed, one can show easily that in the classifier $P^P$, every morphism can be factored as a permutation of colors followed by an order-preserving morphism. This factorization can be lifted to $\cop$ via the morphism of classifiers induced by $P+A \to P$. So in $\cop$, we have two  subcategories: of order-preserving morphisms and of certain isomorphisms generated by permutations of $K_j.$ Using this factorization one can reduce any element in $\Pi_1(\cop)$ to a triangle like

\begin{align*}
\xymatrix{
b \ar[rr] \ar[dr] & & a \ar[dl]\\
& c & 
}
\end{align*}

where the top morphism is an isomorphism induced by permutations of $K_j$, the two other morphisms are order preserving, and $c$ is a local terminal object in the order-preserving subcategory of $\cop$.  If we assume that the induced element in $\Sigma_n$ is trivial then from the terminality of $c$ it follows that the triangle commutes. So, a non-trivial element of $\Pi_1$ yields a non-trivial $\Sigma_n$-action, which we can then convert into trivial cofibration whose pushout is not a weak equivalence, just as in Example 2.9 of \cite{batanin-white-eilenberg-moore}.

\end{remark}

\subsection{Examples of quasi-tame substitudes} \label{subsec:qt-examples}

Quasi-tame substitudes are much more numerous than tame substitudes. In this section, we prove that the polynomial monads for various Grothendieck constructions of Section \ref{sec:Gr(T)} are quasi-tame, so that transferred model structures exist. We begin by proving that every tame substitude is quasi-tame. 

\begin{proposition}\label{proptameimpliesquasitame}
Let $(P,A)$ be a tame substitude. Then $(P,A)$ is quasi-tame.
\end{proposition}

\begin{proof}
Since $(P,A)$ is tame, it has a discrete final subcategory. It follows that the fundamental groupoid consists of a discrete set of points, because $\cop$ is a coproduct of categories with terminal object. Each connected component of $\cop$ is contractible, so $(P,A)$ is quasi-tame.
\end{proof}

For tame $(P,A)$, the proof of Theorem \ref{thm:quasi-tame-admissible}, more precisely of the finality of $\db$, simplifies. We can let $C = \qa_c^{(k)}$ and let $C'$ be the discrete final subcategory guaranteed by tameness. In $C'$, all $X$-vertices have already been multiplied together, so $C'[X^{-1}]$ is isomorphic to $C'$. Since $C'[F^{-1}]$ is a contractible groupoid, it follows that $\bar{\Delta'}: C' \to C'[F^{-1}] \times C'[X^{-1}]$, taking $x\mapsto (x,x)$, is final. 

The monad for monoids is tame, hence quasi-tame. Its quasi-tameness can also be seen directly, following the logic above, since $C[X^{-1}] \cong C$ by inspection of the free monoid filtration. The case for non-symmetric operads is similar, but with alternating trees instead of alternating letters. In this setting, $C'[M^{-1}]$ is a coproduct of indiscrete categories.

The monad for nonreduced symmetric operads is polynomial but is not quasi-tame, because the presence of objects with nontrivial $\Sigma_2$-automorphisms creates non-contractible loops in $\pi_1(\cop)$. 
In practice, this means that, in the terminology of (\ref{diagram:quasi-tame finality triangle}), to compute the colimit of $\XX$, one would need $\XX_f \otimes_{\Sigma_2} \XX_x$.

We have seen several examples of quasi-tame polynomial monads in Section \ref{sec:quasi-tame} and above. We expect to see many more examples in future papers. The following is our main tool for proving quasi-tameness in this paper.

\begin{definition}
	We will say that a morphism of $\Sigma$-free substitudes $i: (T,A) \hookrightarrow (G,U)$ realizes a Grothendieck construction for $F: B \to \Cat$ if:
	\begin{enumerate}
		\item $\int F = \Alg_{(G,U)}$, $B= \Alg_{(T,A)}$,
		\item the canonical projection $\int F \to B$ is isomorphic to the restriction functor $i^*: \Alg_{(G,U)} \to \Alg_{(T,A)}$ along $i$,
		\item $i$ has a retraction $r: (G,U) \to (T,A)$.
	\end{enumerate}
\end{definition}

\begin{theorem}\label{conjecture}
	Assume that a morphism of $\Sigma$-free substitudes $i: (T,A) \hookrightarrow (G,U)$ realizes a Grothendieck construction for $F: B \to \Cat$ and
	\begin{enumerate}
		\item $(T,A)$ is quasi-tame,
		\item the functor $U \to A$ given by the retraction $r$ can be factorized as follows
		\[
			\xymatrix{
				U \ar[rr] \ar[rd] && A \\
				& A+A \ar[ru]_{\nabla}
			}
		\]
		where $\nabla$ is the folding morphism
		\item the square of polynomial monads
		\[
			\xymatrix{
				\Gg + \Uu \ar[r] \ar[d] & \Gg \ar[d] \\
				\Tt + 2\Aa \ar[r] & \Tt
			}
		\]
		induces a faithful functor
		\begin{equation}\label{equivalencegroupoids}
			\Gg^{\Gg + \Uu}  \to r^* \left( \Tt^{\Tt + 2\Aa} \right)
		\end{equation}
	\end{enumerate}
	Then $(G,U)$ is quasi-tame as well.
\end{theorem}

\begin{proof}
	Using the fact that $(T,A)$ is quasi-tame, we get from Lemma \ref{PAA} that $\Pi_1(\Tt^{\Tt+2{\Aa}})$ is equivalent to a discrete groupoid. We deduce that $\Pi_1( \Gg^{\Gg + \Uu})$ is also equivalent to a discrete groupoid since the functor \ref{equivalencegroupoids} is faithful.
\end{proof}

We saw in Section \ref{sec:Gr(T)} that the monad for $Gr(NOp(I))$ is not tame, where $I$ is a non-empty set of colors. Fortunately, this monad is quasi-tame.

\begin{proposition}\label{propositiongrnopquasitame}
The substitude for pairs $(O,A)$, where $O$ is an $I$-colored non-symmetric operad and $A$ is an $O$-algebra, is quasi-tame.
\end{proposition}

\begin{proof}
	Let $(G,U)$ be the substitude for pairs $(O,A)$, where $O$ is an $I$-colored non-symmetric operad and $A$ is an $O$-algebra. We want to prove $(G,U)$ is quasi-tame. Let $(T,A)$ be the substitude for $I$-colored non-symmetric operads. We will prove that the conditions of Theorem \ref{conjecture} are satisfied.
	
	Since $(T,A)$ is tame \cite[Section 9.2]{batanin-berger}, it is also quasi-tame thanks to Proposition \ref{proptameimpliesquasitame}.
	
	Recall that the categories $A$ and $U$ are discrete and given by the sets $PBq(I)$ and $PBq(I) \sqcup I$ respectively, where $PBq(I)$ is the set of planar $I$-bouquets. The morphism $U \to A + A$ is given by the inclusion of $I$ in $PBq(I)$ which sends a color $i$ to the corolla with no inputs and the unique edge colored by $i$.
	
	Now let us prove that the last condition of Theorem \ref{conjecture} is satisfied. To do this, we need to describe more explicitly the morphism between algebras induced by the commutative square of polynomial monads. The category $\Gg^{\Gg + \Uu}$ has as objects trees with $X$-vertices or $K$-vertices, possibly boxed if they have valency one. The category $\Tt^{\Tt+2 \Aa}$ has as objects trees with $X$-vertices, $K$-vertices or $L$-vertices, all circled. The induced functor turns all the boxed $X$-vertices to circled $X$-vertices and boxed $K$-vertices to circled $L$-vertices. It leaves the circled vertices unchanged. Let $C$ be the full subcategory of $\Tt^{\Tt + 2 \Aa}$ of trees whose $L$-vertices have no inputs. Then the induced functor can be factorized
	\begin{equation}\label{factorisation}
		\xymatrix{
			\Gg^{\Gg+\Uu} \ar[rr] \ar[rd] && r^* \left( \Tt^{\Tt + 2 \Aa} \right) \\
			& C \ar@{^{(}->}[ru]
		}
	\end{equation}
	
	The first functor in the factorization has a left adjoint, which turns $L$-vertices to boxed $K$-vertices and leaves the other vertices unchanged. The unit is given by the identity morphisms and the counit is given by the morphisms \ref{morphismzerocircletozerobox}.

	Let $D$ be the full subcategory of $\Tt^{\Tt + 2 \Aa}$ of trees which contain at least one $L$-vertex with at least one input. Then $\Tt^{\Tt + 2 \Aa}$ is the disjoint coproduct $C \sqcup D$. The second functor in the factorization is the inclusion of the form $C \hookrightarrow C \sqcup D$. Therefore the composite of these two functors is indeed faithful, which concludes the proof.
\end{proof}

\begin{proposition}\label{propositionlmodquasitame}
	The substitude for pairs $(O,L)$, where $O$ is an $I$-colored non-symmetric operad and $L$ is a left $O$-module, is quasi-tame.
\end{proposition}

\begin{proof}
	The proof is very similar to the proof of Proposition \ref{propositiongrnopquasitame}. This time we will define the full subcategory $C$ of $\Tt^{\Tt + 2 \Aa}$ of trees whose $L$-vertices are all on top. Again, we will have a factorization as in \ref{factorisation}. The left adjoint of the first functor in the factorization turns $L$-vertices to boxed $K$-vertices and adds a boxed unary $X$-vertex above all the top $X$-vertices.

	The unit is given from the units of the operad, that is the morphisms which insert a circled unary vertex with label $X$. The counit is given from the left module actions, as in \ref{morphismalgebra}.

\end{proof}

Following the same strategy, we can prove the following.

\begin{proposition} \label{prop:quasi-tame-bimodules}
	The substitute for triples $(A,B,C)$ where $A$ and $B$ are non-symmetric operads and $C$ is an $A-B$-bimodule is quasi-tame.
\end{proposition}

\begin{proof}
	Let $(G,U)$ be the substitude for such triples. We want to prove that $\Pi_1( \Gg^{\Gg + \Uu})$ is equivalent to a discrete groupoid. Let $C$ be the full subcategory of $\Gg^{\Gg + \Uu}$ of objects from which there can be no non-trivial morphisms coming from the operad $B$ or the right $B$-module action. There is an inclusion functor from $C$ to $\Gg^{\Gg + \Uu}$ and this inclusion functor has a left adjoint which automatically applies all the morphisms coming from the operad $B$ or the right $B$-module action.
	
	Now let $(G',U')$ be the substitude for pairs $(O,L)$ where $O$ is a non-symmetric operad and $L$ is a left module. Each connected component of this category $C$ is canonically isomorphic to a connected component of $\Gg'^{\Gg' + \Uu'}$, since the morphisms in $C$ are only the morphisms from the operad $A$ and the left $A$-module action. We get the conclusion from Proposition \ref{propositionlmodquasitame}.
\end{proof}

We will explore consequences of these results in the next section. We turn now to the case of symmetric operads. Our techniques are well-suited to the study of constant-free symmetric operads and their modules, a setting that has been applied in work related to the bar construction for strongly dualizable $n$-categories, among other places \cite{batanin-markl}.

\begin{proposition} \label{prop:constant-free-qt}
The substitude for pairs $(O,L)$, where $O$ is an $I$-colored constant-free symmetric operad and $L$ is a constant-free left $O$-module, is quasi-tame.
\end{proposition}

\begin{proof}
	The proof is completely analogous to the proof of Proposition \ref{propositionlmodquasitame} and is based on the fact that the substitude for $I$-colored constant-free symmetric operads is tame \cite[Section 9.4]{batanin-berger}, therefore also quasi-tame.
\end{proof}

Recall \cite{batanin-berger} that an $n$-operad in an operad in the operadic category of $n$-ordinals. An $n$-operad $A$ is \emph{constant-free} if $A_0 = \varnothing$ where $0$ is the initial $n$-ordinal.

\begin{proposition} \label{prop:cf-quasi-tame}
	The substitude for pairs $(O,L)$, where $O$ is a constant-free $n$-operad and $L$ is a constant-free left $O$-module, is quasi-tame.
\end{proposition}

\begin{proof}
	Recall \cite[Proposition 12.17]{batanin-berger} that the polynomial monad for constant-free $n$-operads is generated by the polynomial
	\[
		\xymatrix{
			nROrd & nPTr^*_{reg} \ar[r] \ar[l] & nPTr_{reg} \ar[r] & nROrd
		}
	\]
	where $nROrd$ is the set of isomorphism classes of regular, that is different from initial, $n$-ordinals. The set $nPTr_{reg}$ is the set of isomorphism classes of regular $n$-planar trees, that is the $n$-ordinals decorating each vertex is regular. As usual, $nPTr^*_{reg}$ is the set of trees of $nPTr_{reg}$ with one vertex marked. This polynomial monad is tame \cite[Theorem 12.28]{batanin-berger}. Again, the proof is very similar to the proof of Proposition \ref{propositiongrnopquasitame}. This time the substitudes $(G,U)$ and $(T,A)$ will be the substitudes for pair $(O,L)$, where $O$ is a constant-free $n$-operad and $L$ is a constant-free left $O$-module, and for constant-free $n$-operads respectively. We will have a triangle \ref{factorisation}, where the first morphism of the factorization has a left adjoint. The unit of the adjunction is given thanks to the unit of the operad and the counit is given thanks to the left module action.
\end{proof}

We conclude with an example of a non-quasi-tame polynomial monad.

\begin{example}
Consider the monad $\cat P = Gr(SOp(J))$ for the Grothendieck construction of (any flavor of) symmetric operads, and their algebras. This monad is not quasi-tame. If it were, then taking $\M = Ch(\mathbb{F}_2)$, we would obtain a full model structure on $\int \Phi$, which would imply a full model structure on $\Phi(O)$ for any $O$. In particular, there would be a vertical model structure on commutative differential graded algebras, contradicting \cite[Section 5.1]{white-commutative-monoids}. 

We can also see directly that this monad is not quasi-tame, analogously to \cite[Section 9]{batanin-berger}. In the internal algebra classifier for $Gr(\cat P)$, there will be objects with, for example, nontrivial $\Sigma_2$-automorphisms, as already mentioned in Subsection \ref{subsec:qt-examples}.
Hence $\Pi_1(\cat P^{\cat P+\cat A})$ cannot be equivalent to a discrete groupoid.
\end{example}

\section{Applications} \label{sec:applications}

In this section, we provide a few applications of the previous section.

\subsection{Operads}

Our motivating example takes $\cat B$ to be a category of operads (any flavor given by a $\Sigma$-free substitude \cite{Reedy-paper}) and $\Phi(O)$ to be the category of $O$-algebras or left $O$-modules. In \cite[Theorem 3.9]{companion}, we prove that, if $\int \Phi$ admits the global model structure, then the categories $\cat B$ (e.g., of monoids or operads) and $\Phi(O)$ (of algebras and modules) admit horizontal and vertical model structures, that are furthermore (relatively) left proper if the ambient model category $\M$ satisfies the conditions of Theorem \ref{thm:quasi-tame-admissible}. Furthermore, we prove a rectification result, regarding when a weak equivalence $\phi: O\to O'$ induces a Quillen equivalence between $\Phi(O)$ and $\Phi(O')$. The following is an immediate consequence of the work of the previous section.

\begin{theorem} \label{thm:global-model-examples}
Assume $\M$ is a compactly generated monoidal model category satisfying the monoid axiom. Then the following global model structures exist:
\begin{enumerate}
\item The category of pairs $(R,M)$ where $R$ is a monoid (or, more generally, an $A$-algebra for a commutative monoid $A$) and $M$ is an $R$-module (left, right, or bimodule).
\item The category of pairs $(O,A)$ where $O$ is a non-symmetric operad and $A$ is an $O$-algebra.
\item The category of pairs $(O,M)$ where $O$ is a non-symmetric operad and $M$ is a left $O$-module.
\item The category of pairs $((O,P),M)$ where $O$ and $P$ are non-symmetric operads, and $M$ is an $O-P$-bimodule.
\item The category of pairs $(O,M)$ where $O$ is a constant-free symmetric operad (meaning $P(0)=\emptyset$ is the initial object of $\M$) and $M$ is a constant-free left $O$-module.
\item The category of pairs $(O,M)$ where $O$ is a constant-free $n$-operad and $M$ is a constant-free left $O$-module.
\end{enumerate}

If $\M$ is $h$-monoidal then these model structures are relatively left proper. If $\M$ is strongly $h$-monoidal then they are left proper.
\end{theorem}

\begin{proof}
This follows from Theorem \ref{thm:quasi-tame-admissible} and Propositions \ref{prop:Gr(Mon)-tame}, \ref{propositiongrnopquasitame}, \ref{propositionlmodquasitame}, \ref{prop:quasi-tame-bimodules}, \ref{prop:constant-free-qt}, and \ref{prop:cf-quasi-tame}.
\end{proof}

Thanks to \cite[Theorem 3.9]{companion} we obtain model structures on the category of non-symmetric operads and on categories of algebras and modules over them, and a rectification result: if $P$ and $O$ are non-symmetric operads such that $P(n)$ and $O(n)$ are cofibrant in $\M$ for all $n$, then a weak equivalence $f: P \to O$ induces a Quillen equivalence on algebras. These results recover results from \cite{muro}, and extend them to the case of many-colored non-symmetric operads.

If $\M$ is compactly generated and satisfies the monoid axiom (resp. strongly $h$-monoidal), then the global model structure is relatively left proper, relative to entrywise cofibrant non-symmetric operads (resp. left proper), by Theorem \ref{thm:quasi-tame-admissible}. In \cite[Theorem 3.18]{companion}, we prove that this implies the category of non-symmetric operads and the categories of $O$-algebras or left $O$-modules, are relatively left proper (resp. left proper). This improves on \cite[Corollary 8.4]{muro}, which required all objects in $\M$ to be cofibrant, in order to deduce left properness.

Another application is to the theory of bimodules, which was important in recent work of Turchin and Dwyer-Hess on the space of long knots \cite{turchin,dwyerhess}. The category of $(P,Q)$-bimodules, where $P$ and $Q$ are non-symmetric operads, can itself be viewed as a category of algebras over a colored non-symmetric operad. Our work provides the first left proper model structure for the category of $(P,Q)$-bimodules over a general base model category $\M$ (satisfying our usual hypotheses). Of course, it is also possible to study operadic bimodules globally via the Grothendieck construction, since Proposition \ref{prop:quasi-tame-bimodules} implies the global model structure exists. The global perspective provides change-of-operad results via the rectification machinery of \cite[Theorem 3.22]{companion}.

We explore the case of symmetric operads in \cite[Section 4]{companion}, where we prove, for example, that the category of $(P,Q)$-bimodules has a transferred semi-model structure under cofibrancy conditions on $P$ and $Q$, proving an analogue of \cite[Theorem 16.2.A]{fresse-book}. 
For topological symmetric operads this can be upgraded to a full model structure, since all colored operads in topological spaces are admissible.

Our last example of an important category of algebras over a colored non-symmetric operad is the category of infinitessimal $O$-bimodules where $O$ is a non-symmetric operad. Hence, we obtain a model structure on this category, useful to \cite{deleger}.

\subsection{Opetopic sequences}

Recall that every polynomial monad $T$ has an associated monad $T^+$, defined via the Baez-Dolan plus construction \cite[Section 11]{batanin-berger}. For example, if $T = NOp$ then $Alg_{T^+}$ consists of hyperoperads. In Section \ref{subsectionhigherextensions} we prove that there is a polynomial monad $Gr(T^+)$ for the category of pairs $(O,A)$ where $O$ is a $T^+$-algebra and $A$ is an $O$-algebra.

As a consequence, we may study the following \textit{opetopic sequence}:
\[
Id, Mon, NOp, NOp^+, \dots
\]

An \textit{opetopic fibration} is one where the base comes from the sequence above. We next discuss quasi-tameness, and hence transferred model structures, associated with an opetopic fibration. For example, with $NOp^+$, we obtain a model structure on pairs $(O',O)$ where $O$ is a non-symmetric operad. 

\begin{proposition}\label{prop:GrT+ is qt}
	If $P$ is a polynomial monad in the opetopic sequence, then $Gr(P)$ is quasi-tame.
\end{proposition}

\begin{proof}
	If $P=Id$ or $P=Mon$, the result is given by Subsection \ref{subsec:monoids-modules}. Let us now write $P^{++} = (P^+)^+$ and prove that $Gr(P^{++})$ is quasi-tame. Recall from \cite[Subsection 3.4]{batanin-kock-joyal} that the polynomial for $P^{++}$ is given by
	\[
		\xymatrix{
			tr(B) & tr^*(P^+) \ar[l] \ar[r] & tr(P^+) \ar[r] & tr(B)
		}
	\]
	where $tr(B)$ is the set of isomorphism classes of rooted trees decorated by elements of the set of operations $B$ of $P$, $tr(P^+)$ is the set of isomorphism classes of rooted trees whose vertices are decorated by elements of $tr(B)$ and edges are decorated by elements of $B$. The set of colors of $Gr(P^{++})$ is $tr(B) \sqcup B$. Its set of operations is the set of isomorphism classes of rooted trees whose vertices are decorated by elements of $tr(B)$ and edges are decorated by elements of $B$, and some vertices can be boxed if they are decorated with the trivial tree without any edges. As before we can apply Theorem \ref{conjecture}, taking $(T,A)$ to be the substitute for $P^{++}$-algebras and $(G,U)$ to be the substitute for $Gr(P^{++})$-algebras. The induced functor \ref{equivalencegroupoids} turns all the boxed $X$-vertices to circled $X$-vertices and turns boxed $K$-vertices to circled $L$-vertices. The rest of the proof is similar to the proof of Proposition \ref{propositiongrnopquasitame}.
\end{proof}

We envision several future applications of this result. For example, we hope this result can be used to prove a generalization of the Turchin/Dwyer-Hess double-delooping theorem \cite{turchin,dwyerhess}, where we replace the category of non-symmetric operads by the category of algebras for a polynomial monad in the opetopic sequence.

\subsection{Homotopy T-algebras}

The global perspective provides another point of view for the study of homotopy $T$-algebras for a given polynomial monad $T$. Recall \cite[Section 11]{batanin-berger} that if $T$ is a polynomial monad then the category of \textit{homotopy $T$-algebras} is constructed as $Q\tau$-alg where $Q\tau$ is the cofibrant replacement of the terminal $T^+$-algebra $\tau$. We know that $Alg_\tau \cong Alg_T$ has a transferred semi-model structure under very general conditions because $T$ is $\Sigma$-cofibrant as a $J$-colored operad \cite[Theorem 6.3.1]{white-yau}. We can take the cofibrant replacement $Q\tau$ for $\tau$ in the model category $Alg_{T^+}$ (using that $T^+$ is tame, for any polynomial monad $T$). Algebras over this $Q\tau$ are defined to be $Alg_{\phi_!(Q\tau)}$. Since $\phi_!$ is a left Quillen functor, and $Q\tau$ is cofibrant, $\phi_!(Q\tau)$ is cofibrant in $SOp(Bq(J))$. Thus, $Alg_{\phi_!(Q\tau))}$ has a transferred semi-model structure, claimed to be a full model structure in Theorem 4 of \cite{spitzweck-thesis} if the ambient model category $\M$ satisfies the monoid axiom. This provides a form of an answer to the Batanin-Berger Conjecture \cite[Section 11]{batanin-berger}, and a powerful tool to study homotopy $T$-algebras.

Another consequence of Proposition \ref{prop:GrT+ is qt} involves rectification, as \cite[Theorem 3.22]{companion} yields a Quillen equivalence between $Alg_{Q\tau}$ and  $Alg_{\tau}= Alg_T.$ In particular any homotopy $T$-algebra is equivalent to a strict $T$-algebra. We emphasize that $T$ is a polynomial monad (hence $\Sigma$-cofibrant).

\subsection{Twisted modular operads}

In this section, we summarize an application of the global model structure proven in \cite[Section 5]{companion}. The theory of twisted modular operads dates back to \cite{getzler-kapranov}. The setting is the projective model structure $Ch(k)$ on $\mathbb{Z}$-graded chain complexes of $k$-vector spaces where $k$ is a field of characteristic zero. In \cite[Section 4]{getzler-kapranov}, the authors construct a hyperoperad whose algebras are modular operads. For a fixed cocycle $D$, they define an object, which we denote $O$, such that $O$-algebras are twisted modular operads. Because there is no known Set-valued operad whose algebras are twisted modular operads, standard techniques cannot be used to transfer a model structure to the category of twisted modular operads. However, the Grothendieck construction provides an alternative approach. The following is proven in \cite[Theorem 5.1]{companion}:

\begin{theorem}
The category of twisted modular operads \cite{getzler-kapranov} admits a vertical model structure.
\end{theorem}

The idea of the proof is to transfer a model structure to the category $\int \Phi$ whose objects are pairs $(H,O)$ where $H$ is a hyperoperad and $O$ is an algebra over $H$. Even though we don't know that the relevant polynomial monad $Gr(P)$ is quasi-tame, the fact that $k$ has characteristic zero tells us that $Gr(P)$ is admissible. Since the global model structure exists, so does each vertical model structure. Since it does not require quasi-tameness, this result is an application of Proposition \ref{prop:poly-for-Gr} more than Section \ref{sec:quasi-tame}. As far as the authors are aware, this is the first non-trivial model structure on the category of twisted modular operads.

\section{Grothendieck construction for commutative monoids} \label{sec:non-poly}

In this section, we illustrate that our techniques can also be used for non-polynomial monads, at least in certain special cases. For a cofibrantly generated monoidal model category $\M$, let $\CMon(\M)$ denote the category of commutative monoids and define a functor $\Phi$ such that for a given commutative monoid $R$, $\Phi(R)$ is the category of left $R$-modules. We first discuss transferred (semi-)model structures on $\CMon(\M)$ and then on $\int \Phi$.

\subsection{Commutative monoids}

Say that $\M$ satisfies the \textit{commutative monoid axiom} if, whenever $f$ is a trivial cofibration then $f^{\boxprod n}/\Sigma_n$, obtained from the $n$-fold iterated pushout product, is a trivial cofibration. For such $\M$, $\CMon(\M)$ has a transferred semi-model structure, which is a model structure if $\M$ also satisfies the monoid axiom \cite[Definition 3.1, Theorem 3.2, Corollary 3.8]{white-commutative-monoids}. This result is proven using a filtration that the author constructed directly, that realizes the critical pushout (\ref{diagram:pushout}) as a transfinite composition of pushouts of morphisms of the form $X\otimes f^{\boxprod n}/\Sigma_n$. Our first step is to show how this filtration is related to the theory of classifiers \cite{batanin-berger}, so that we can generalize it for the global model structure.

The free commutative monoid monad $Com$ is not polynomial. It is, however, possible to develop a theory of internal algebras and their classifiers for monads coming from arbitrary colored symmetric operads in $\Set$, following \cite{batanin-berger}. Despite the fact that such a monad is not necessary cartesian, we can apply a different construction which produces a polynomial operad in $\Cat$ whose internal algebras behave in similar way as for the cartesian case. Calculations in general are harder, since they involve a calculation of nondiscrete codescent objects (e.g., due to the non-free $\Sigma_n$-actions). However, there is a procedure for how to do such calculations \cite{Weber}.  It is well-known fact (which can be obtained also from general theory of classifiers \cite{Weber}) that the symmetric monoidal category of finite sets is a classifier for commutative monoids.

When translated to the realm of classifiers, the explicit calculations of \cite{white-commutative-monoids}, to describe the critical pushout (\ref{diagram:pushout}), enable a decomposition of the classifier $Com^{Com_{f,g}}$ in a way very analogous to the theory of quasi-tame polynomial monads. We now explicitly describe the relevant classifiers, referring the reader to \cite{Weber} for full details.

As in the polynomial case, $Com^{Com+1}$ computes semifree coproducts of commutative monoids, which are, of course, just tensor products of $X$ and $\Sym(K).$ The objects of this classifier are finite sets over two element set $\{X,K\}$, and the morphisms are any morphisms of these two-colored sets that are bijections on $K.$ It is not hard to see that this category has a final subcategory consisting of objects with exactly one element with $X$ color. This is, actually, a groupoid which is equivalent to the groupoid of symmetric groups.

The classifier $Com^{Com_{f,g}}$ can be constructed in a similar way. Its objects are finite sets with three colors $X,K,L.$ Morphisms are again morphisms of colored sets which are bijections on $K$ and $L$ elements. There are also two generators $g:K\to X$ which simply change the color of an element from $K$ to $X$ and similarly $f:K\to L.$ 

We have a final subcategory of $Com^{Com_{f,g}}$ which consists in this case of finite sets with a single $X$-colored element. This final subcategory has a filtration $\bf t^{(n)}$ by the number of elements with $K$ and $L$ colors (that is, just a cardinality of the set minus one). And so, the colimit over $Com^{Com_{f,g}}$ can be computed as a sequential colimit over colimits over $\bf t^{(n)}$ like in the non-commutative case. The difference with the non-commutative case is that colimit over $\bf t^{(n)}$ cannot be computed as a pushout of a colimit over punctured cubes. In the commutative case, the relevant cube has a nontrivial group of automorphisms equal to $\Sigma_n.$ So, we need to saturate the class of morphisms $X\otimes f^{\boxprod n}/\Sigma_n$ where $f$ runs over trivial cofibrations. We need then the commutative monoid axiom to be sure that result of such a colimit is a weak equivalence. 

\subsection{Commutative monoids and their modules}

It is possible to generalize Proposition \ref{prop:poly-for-Gr} to $\Cat$-valued polynomials. We are most interested in the case of $\int \Phi$ where $\Phi(R)$ is the category of $R$-modules for a commutative monoid $R$. 

The monad $Gr(Com)$ is determined by a symmetric operad in $\Set$ with two colors $r$ and $m.$ The corresponding polynomial monad in $\Cat$ is 
\begin{diagram}\{r,m\} &\lTo^{s}&D^*&\rTo^{p}&D&\rTo^{t}&\{r,m\}\end{diagram}    
Here, $D$ is the category whose objects are finite sets and pointed finite sets. The morphisms are bijections which preserve points (so there are no morphisms between pointed and unpointed sets).  The target of an unpointed set is $r$ and the target of a pointed set is $m.$ The category $D^*$ has objects the objects of $D$ with one point marked, and morphisms are morphisms of $D$ which preserve the marked points.  The source of such an object is $r$ if the marked point is not a distinguished point. Otherwise the source is $m.$ The substitution operation is just replacing a marked element by the set which we want to substitute in an obvious sense. 

It is not hard to see that the algebras of such a polynomial monad are pairs $(\cat R,\cat C)$ where $\cat R$ is a strict symmetric monoidal category and $\cat C$ is a category equipped with a strict action of $\cat R$ on $\cat C.$ A pseudoalgebra of this monad consists of a   symmetric  monoidal category and a category on which it acts in a pseudo sense. The internal algebras in a pseudoalgebras $(\cat R,\cat C)$ of this monad is a commutative monoid $r$ in $\cat R$ together with an object  $m\in \cat C$ and an action $r\otimes m \to m$ where $\otimes$ is the action of $\cat R$ on $\cat C$ subject natural axioms (Baez and Dolan called this action riding the action of $\cat R$ on $\cat C$). In particular, for $(\cat R,\cat C)=(\M,\M)$ such an internal algebra is simply a commutative monoid $R$ in $\M$ together with an $R$-module. 

The techniques of Section \ref{sec:Gr(T)} can now be used to produce the classifier $Gr(Com)^{Gr(Com)+1}$. Just like the cases of monoids and non-symmetric operads, it looks like the classifier $Com^{Com+1}$ but with boxes as well, corresponding to the pointed finite sets in $D$. This classifier has a final subcategory consisting of finite sets and finite pointed sets with a single $X$-colored element. The classifier $Gr(Com)^{Gr(Com)_{f,g}}$ has a similar final subcategory and, just like the case of $Com$, this final subcategory has a filtration based on the number of elements with $K$ and $L$ colors. The critical pushout (\ref{diagram:pushout}), now where $T$ is the free $Gr(Com)$-functor, may now be computed as a sequential colimit as in Proposition \ref{filtration}.
We arrive at the following result, where we use $I^{\otimes}$ to denote the monoidal saturation of the class of cofibrations, i.e., the class of transfinite compositions of pushouts of morphisms of the form $X\otimes f$ where $X$ is any object and $f$ is a cofibration.

\begin{theorem} \label{thm:Gr(Com)-model}
Suppose $\M$ is a cofibrantly generated monoidal model category satisfying the commutative monoid axiom, where the domains of the generating (trivial) cofibrations are small relative to $I^{\otimes}$. Then the Grothendieck construction $\int \Phi$, whose objects are pairs $(R,A)$ where $R$ is a commutative monoid and $A$ is an $R$-module, inherits a transferred semi-model structure which is a full model structure if $\M$ satisfies the monoid axiom.
\end{theorem}

\begin{proof}
The free functor $T:\M \times \M \to \int \Phi$, defined by $T(X,Y) = (\Sym(X),\Sym(X)\otimes Y)$ is left adjoint to the forgetful functor $U:\int \Phi \to \M \times \M$, since $\Sym(X)\otimes Y$ is the free $\Sym(X)$-module on $Y$.

The critical pushout starts with a trivial cofibration $(f_1,f_2):(K_1,K_2)\to (L_1,L_2)$ in $\M \times \M$, and an attaching morphism $(g_1,g_2):(K_1,K_2)\to U(R,A)$. One then takes the following pushout in $\int \Phi$:
\begin{align} \label{diagram:pushout-Gr(Com)}
\xymatrix{ \po
T(K_1,K_2) \ar[r]^{(\beta,b)} \ar[d]_{(\alpha,a)} & T(L_1,L_2) \ar[d]^{(\psi,h)} \\
(R,A) \ar[r]_{(\phi,f)} & (P,B) \\
}
\end{align}
One must show that transfinite compositions of morphisms $(\phi,f):(R,A)\to (P,B)$, obtained by such pushouts, are weak equivalences in $\int \Phi$. 
\cite[Lemma 3.23]{companion} shows how to compute such pushouts. The classifier filtration described above allows us to calculate $(\phi,f)$ as a sequential colimit of pushouts in $\M \times \M$. On the first component, the morphisms in the sequential colimit are pushouts of morphisms of the form $R\otimes f_1^{\boxprod n}/\Sigma_n$, so the commutative monoid axiom guarantees that $\phi$ (and transfinite compositions of such morphisms) is a weak equivalence as soon as either $R$ is cofibrant or $\M$ satisfies the monoid axiom. We turn to the second component. By \cite[Lemma 3.23]{companion}, we must compute the following pushout in $\Sym(K_1)$-mod:
\[
\xymatrix{ \po
\Sym(K_1)\otimes K_2 \ar[r] \ar[d]^{a} & \phi^*(\Sym(L_1) \otimes L_2) \ar[d] \\
\alpha^*(A) \ar[r] & \alpha^* \phi^* B = \beta^* \psi^* B\\}
\]

When we compute this pushout in $\M$, we conclude that $f$ is a trivial cofibration, thanks to \cite[Lemma A.3]{white-commutative-monoids} and hence transfinite compositions of morphisms of the form $(\phi,f)$ arising from (\ref{diagram:pushout-Gr(Com)}) are all weak equivalences, hence \cite[Lemma 2.3]{SS00} (resp. \cite[Theorem 2.2.1]{Reedy-paper}) gives the model (resp. semi-model) structure.
\end{proof}

\begin{remark}
To better understand (\ref{diagram:pushout-Gr(Com)}), it is helpful to factor the pushout square into a composite of two easier pushout squares:

\[
\xymatrix{
(\Sym(K_1),\Sym(K_1)\otimes K_2) \ar[r]^{(\beta,b)} \ar[d]^{(\alpha,\eta)} & (\Sym(L_1),\Sym(L_1)\otimes L_2) \ar[d]^{(\psi,\eta)} \\ 
(R,\phi_!(\Sym(K_1)\otimes K_2)) \ar[r] \ar[d]_{(id,{}^\dashv a)} & (P,\psi_!(\Sym(L_1)\otimes L_2)) \ar[d] \\
(R,A) \ar[r]_{(\phi,f)} & (P,B)\\
}
\]
Since the outer rectangle is a pushout, the bottom square will be a pushout as soon as the top square is proven to be one. This is an easy exercise.

The top pushout is essentially happening in $\CMon(\M)$ and the bottom pushout is essential happening in a module category (it can be computed either in $R$-mod or $P$-mod, by \cite[Lemma 3.23]{companion}. 

The top pushout can be handled by the techniques of \cite{white-commutative-monoids}, since the second component (the module part) plays no role. 
Similarly, analysis of the bottom pushout reduces to second component, which is straightforward, using the methods of Theorem 4.1 of \cite{SS00}. Essentially, this comes down to checking that the monad $T_{\Sym(K_1)}(M) = \Sym(K_1)\otimes M$ has $J_{T_{\Sym(K_1)}}$-cofibrations contained in the weak equivalences. When $\Sym(K_1)$ is a cofibrant commutative monoid and $M$ is cofibrant, this is always true. In general, this is true when the monoid axiom holds.
\end{remark}

\begin{remark}
The proof of Lemma A.1 of \cite{white-commutative-monoids} can be mimicked to prove that it suffices to check (\ref{diagram:pushout-Gr(Com)}) for the \textit{generating} trivial cofibrations $j$ of $\M \times \M$.
\end{remark} 

\begin{remark}
In \cite{harpaz-prasma-integrated}, $\M$ was required to be \textit{commutatively flat} in order to produce the global model structure of Theorem \ref{thm:Gr(Com)-model}. This means that, for every $R$ and every $R$-module $M$, the operation $- \otimes_R M$ is required to take weak equivalences of $R$-modules to weak equivalences in $\M$. Our treatment does not require this condition.
\end{remark}

\begin{remark}
The work in this section also proves that, if $A$ is a commutative monoid, then the category of pairs $(R,M)$ where $R$ is a commutative $A$-algebra and $M$ is an $R$-module, inherits a transferred model structure. For this setting, we work with $A$-modules as our base model category and note that, by \cite{SS00} (Theorem 4.1), this category of $A$-modules satisfies the pushout product axiom and monoid axiom.
\end{remark}

\begin{remark}
A similar approach to Theorem \ref{thm:Gr(Com)-model} may be applied to the Swiss-cheese operad and the category of pairs $(\cat B, \cat M)$ where $\cat B$ is a braided category and $\cat M$ is a monoidal category. Even more generally, it applies to pseudoalgebras of a contractible categorical  $2$-operad with two colors.
\end{remark}

\subsection{Left properness for commutative monoids}

With the filtration of the previous section in hand, it is easy to prove that the category $\int \Phi$ of pairs $(R,M)$ is left proper. For this, we must also assume that $\M$ is $h$-monoidal and we must assume the strong commutative monoid axiom, which is the commutative monoid axiom plus the requirement that, whenever $f$ is a cofibration, then $f^{\boxprod n}/\Sigma_n$ is a cofibration. Under these conditions, $\CMon(\M)$ is left proper \cite[Definition 3.4, Theorem 4.17]{white-commutative-monoids}. We now obtain a relative version, and the relevant results for $\int \Phi$. The following theorem complements \cite[Theorem 4.17]{white-commutative-monoids} and extends \cite[Theorem 3.1]{batanin-berger} to the case of commutative monoids.

\begin{theorem} \label{thm:cmon-left-proper}
Let $\M$ be a compactly generated monoidal model category satisfying the strong commutative monoid axiom and the monoid axiom. Then $\CMon(\M)$ is relatively left proper. If, furthermore, $\M$ is strongly $h$-monoidal, then $\CMon(\M)$ is left proper.
\end{theorem}

\begin{proof}
The proof proceeds just like \cite[Theorem 4.17]{white-commutative-monoids}, but fewer hypotheses are needed because we must only analyze weak equivalences between underlying cofibrant objects. They key step in the proof of \cite[Theorem 4.17]{white-commutative-monoids} (where all the hypotheses are used) is to study the following cube, where $f: A\to B$ is a weak equivalence (between underlying cofibrant objects in the relative left properness case) and $u: K\to L$ is a cofibration that we are attaching via $\Sym(u)$, $\alpha: K \to U(A)$, and a pushout like (\ref{diagram:pushout}). The morphisms $A[u]^{(n-1)} \to A[u]^{(n)}$ filter $A\to A[u,\alpha]$, and similarly for $B$.

\begin{align*}
\xymatrix{A\otimes Q_n/\Sigma_n \ar[rr] \ar[dd] \ar[dr] && A \otimes L^{\otimes n}/\Sigma_n \ar[dd] \ar[dr] & \\
& A[u]^{(n-1)} \ar[rr] \ar[dd] && A[u]^{(n)} \ar[dd] \\
B\otimes Q_n/\Sigma_n \ar[rr] \ar[dr] && B \otimes L^{\otimes n}/\Sigma_n \ar[dr] & \\
& B[u]^{(n-1)} \ar[rr] && B[u]^{(n)}}
\end{align*}

In the proof of \cite[Theorem 4.17]{white-commutative-monoids}, one assumes $A[u]^{(n-1)} \to B[u]^{(n-1)}$ is a weak equivalence and must deduce the same for $A[u]^{(n)} \to B[u]^{(n)}$. In the relatively left proper case, \cite[Proposition 2.12]{batanin-berger} allows us to start with a cofibration $u: K\to L$ between cofibrant objects, so all objects are cofibrant in this cube, and the two vertical morphisms in the back face are of the form $f\otimes X$ where $f$ is a weak equivalence between cofibrant objects and $X$ is a cofibrant object (such as $Q_n/\Sigma_n$) by the strong commutative monoid axiom. Such morphisms $f\otimes X$ are always weak equivalences, so the induction of \cite[Theorem 4.17]{white-commutative-monoids} proves the relevant pushout of $f$ is a weak equivalence, as required for relative left properness.

If $\M$ is strongly $h$-monoidal then the morphisms $X\otimes f^{\boxprod n}/\Sigma_n$ are $h$-cofibrations (for $X$ either $A$ or $B$). Since $\M$ is strongly $h$-monoidal, the vertical morphisms $f\otimes X$ are weak equivalences and $\M$ is left proper \cite[Lemma 1.12]{batanin-berger}. Thus, the induced morphism $A[u]^{(n)} \to B[u]^{(n)}$ is a weak equivalence, by \cite[Proposition 1.8]{batanin-berger}.
\end{proof}

We now give the analogous result for the Grothendieck construction $\int \Phi$ whose objects are pairs $(R,M)$ where $R$ is a commutative monoid and $M$ is an $R$-module.

\begin{theorem}
Let $\M$ be a compactly generated monoidal model category satisfying the strong commutative monoid axiom and the monoid axiom. Then the model structure of Theorem \ref{thm:Gr(Com)-model} on $\int \Phi$ is relatively left proper. 
Assume either of the following two conditions:
\begin{enumerate}
\item $\M$ is $h$-monoidal, the domains of the generating cofibrations and the monoidal unit are cofibrant, and cofibrant objects are flat (i.e., $X\otimes -$ preserves weak equivalences).
\item $\M$ is strongly $h$-monoidal.
\end{enumerate}
Then $\int \Phi$ is left proper.
\end{theorem}

\begin{proof}
The proof of the first part is just like the first part of Theorem \ref{thm:cmon-left-proper}, using the filtration for the pushout (\ref{diagram:pushout-Gr(Com)}). Since we are allowed to take $f = (f_1,f_2)$ as a cofibration between cofibrant objects, all objects in the cube are cofibrant, and the vertical weak equivalences follow in the same way (still $f\otimes X$ but now with $X = Q_n/\Sigma_n \otimes K_2$ and similarly for the $L$-part).

The proof of (1) follows exactly as in \cite[Theorem 4.17]{white-commutative-monoids}, but now carrying around extra terms of the form $K_2$ and $L_2$ that do not affect the proof in any way. Again, they can be taken to be cofibrant because of the hypothesis on $\M$.

The proof of (2) follows exactly like in Theorem \ref{thm:cmon-left-proper}. The horizontal morphisms in the back face are of the form $X\otimes u^{\boxprod n}/\Sigma_n \otimes f_2$ where $X$ is $A$ or $B$. Hence, these morphisms are $h$-cofibrations and the rest of the proof goes in the same way.
\end{proof}

In future work, we plan to study left/right Bousfield localizations of the various Grothendieck model structures we have produced in this paper, along with operad-algebra preservation results after \cite{white-localization, batanin-white-eilenberg-moore, white-yau3, white-yau2, white-yau4}, and hence left/right properness will be very valuable for us.

\end{document}